\documentclass{birkjour} 
 \makeatletter
 \def\BlackBoxes{\global\overfullrule5\p@}
 \makeatother
 \BlackBoxes
 \newcommand*{\beq}{$$\refstepcounter{equation}}
 \newcommand*{\eeq}{\eqno(\theequation)$$}
 \newcommand*{\bal}{\begin{aligned}}
 \newcommand*{\eal}{\end{aligned}}
 \newcommand*{\qa}{,\qquad}
 \newcommand*{\qb}{,\quad}
 \newcommand*{\mf}[1]{\boldsymbol{#1}}  
 \newcommand*{\ci}{\mathaccent"7017 }  
 \newcommand*{\hb}[1]{\hbox{$#1$}} 
 \newcommand*{\sdot}{\!\cdot\!}
 \newcommand*{\sco}{\kern2pt\colon\kern2pt}
 \newcommand*{\sn}{\kern1pt|\kern1pt}
 \newcommand*{\bsn}{\kern1pt\big|\kern1pt}
 \newcommand*{\ssm}{\!\setminus\!}
 \newcommand*{\bssm}{\,\big\backslash\,}
 \newcommand*{\vsdot}{\hbox{$|\sdot|$}}
 \newcommand*{\Vsdot}{\hbox{$\|\sdot\|$}}
 \newcommand*{\hh}[1]{{\textbf{#1}}}
 \newcommand*{\npb}{\postdisplaypenalty=10000}
 \newcommand*{\pr}{\hbox{$(\cdot,\cdot)$}}
 \newcommand*{\prsn}{\hbox{$(\cdot\sn\cdot)$}}
 \newcommand*{\pw}{\hbox{$\dl{}\sdot{},{}\sdot{}\dr$}}
 \newcommand*{\ol}{\overline}
 \newsavebox{\Prel}
 \sbox{\Prel}{\begin{picture}(2,6)(0,0)\put(1,2.5){\circle*{2}}\end{picture}}
 \newcommand*{\btdot}{\mathrel{\usebox{\Prel}}}
 \newsavebox{\Ptop}
 \sbox{\Ptop}{\begin{picture}(2,2)(-1,-1)\put(0,0){\circle*{2}}\end{picture}}
 \newcommand*{\thdot}[3]{\mbox{\slap{\kern#3ex\raisebox{#2ex}{\usebox{\Ptop}}}{$#1$}}}
 \newcommand*{\thgK}{\thdot{\gK}{1.9}{.7}}
 \newcommand*{\thgS}{\thdot{\gS}{1.9}{.4}}
 \newcommand*{\thg}{\thdot{g}{1.3}{.4}}   
 \newcommand*{\ithg}{{\thdot{\scriptstyle g}{0.9}{.2}}\vph{\scriptstyle g}} 
 \newcommand*{\thia}{\rlap{\thdot{\ia}{1.3}{.4}}\ph{\ia}}      
 \newcommand*{\thka}{\rlap{\thdot{\ka}{1.3}{.4}}{\ph{\ka}}}    
 \newcommand*{\ithka}{{\thdot{\scriptstyle\ka}{.9}{.4}}}   
 \newcommand*{\thna}{\rlap{\thdot{\na}{1.9}{.65}}{\ph{\na}}}   
 \newcommand*{\thrho}{\rlap{\thdot{\rho}{1.3}{.4}}\ph{\rho}} 
 \newcommand*{\thu}{\thdot{u}{1.3}{.4}}   
 \newcommand*{\ph}{\phantom}
 \newcommand*{\vph}{\vphantom}
 \newcommand*{\cip}{\ci{\vph{a}}\kern.2ex}                        
   
 \newcommand*{\hr}{\hookrightarrow}
 \newcommand*{\Llr}{\Longleftrightarrow}
 \newcommand*{\ra}{\rightarrow}
 \newcommand*{\dl}{\langle}
 \newcommand*{\dr}{\rangle}
 \newcommand*{\slap}[1]{\hspace{0pt}\hbox to 0pt{#1\hss}}
 \newcommand*{\card}{\mathop{\rm card}\nolimits}
 \newcommand*{\tdiv}{\mathop{\rm div}\nolimits}
 \newcommand*{\dom}{\mathop{\rm dom}\nolimits}
 \newcommand*{\gl}{\mathop{\rm gl}\nolimits}
 \newcommand*{\grad}{\mathop{\rm grad}\nolimits}
 
 \newcommand*{\cc}{{\rm cc}}
 \newcommand*{\Lis}{{\mathcal L}{\rm is}}
 \newcommand*{\loc}{{\rm loc}}
 \renewcommand*{\Re}{\mathop{\text{\rm Re}}\nolimits}
 \newcommand*{\is}{\subset}
 \newcommand*{\es}{\emptyset}
 \newcommand*{\iy}{\infty}
 \newcommand*{\mt}{\mapsto}
 \newcommand*{\pl}{\partial}
 \newcommand*{\pa}{\partial^\alpha}
 \newcommand*{\sh}{\sharp}
 \newcommand*{\cona}{\kern-1pt}
 \newcommand*{\coU}{\kern-1pt}
 \newcommand*{\coV}{\kern-1pt}
 \newcommand*{\cotdiv}{\kern-1pt}
 \newcommand*{\coW}{\kern-1pt}
 \newcommand*{\al}{\alpha}
 \newcommand*{\ba}{\beta}
 \newcommand*{\da}{\delta}
 \newcommand*{\Ga}{\Gamma}
 \newcommand*{\ga}{\gamma}
 \newcommand*{\ia}{\iota}
 \newcommand*{\ka}{\kappa}
 \newcommand*{\tk}{{\tilde{\ka}}}
 \newcommand*{\lda}{\lambda}
 \newcommand*{\na}{\nabla}
 \newcommand*{\Om}{\Omega}
 \newcommand*{\om}{\omega}
 \newcommand*{\pO}{{\pl\Omega}}
 \newcommand*{\Sa}{\Sigma}
 \newcommand*{\sa}{\sigma}
 \newcommand*{\ve}{\varepsilon}
 \newcommand*{\vp}{\varphi}
 \newcommand*{\AB}{(\cA,\cB)}
 \newcommand*{\ABs}{(\cA',\cB')}
 \newcommand*{\hAB}{(\hat\cA,\hat\cB)}
 
 \newcommand*{\EF}{(E,F)}
 \newcommand*{\FE}{(F,E)}
 \newcommand*{\IB}{(I,B)}
 
 \newcommand*{\Mg}{(M,g)}
 \newcommand*{\tMtg}{(\tilde{M},\tilde{g})}
 \newcommand*{\Mhg}{(M,\hat g)}
 \newcommand*{\rgK}{(\rho,\gK)}
 \newcommand*{\trtK}{(\tilde{\rho},\tilde{\gK})}
 \newcommand*{\SV}{(S,V)}
 \newcommand*{\hV}{(\hat V)}
 \newcommand*{\Vrho}{(V;\rho)}
 \newcommand*{\XcX}{(X,\cX)}
 \newcommand*{\XY}{(X,Y)}
 \newcommand*{\BB}{{\mathbb B}}
 \newcommand*{\BH}{{\mathbb H}}
 \newcommand*{\BJ}{{\mathbb J}}
 \newcommand*{\BN}{{\mathbb N}}
 \newcommand*{\BR}{{\mathbb R}}
 \newcommand*{\cA}{{\mathcal A}}
 \newcommand*{\cB}{{\mathcal B}}
 \newcommand*{\cD}{{\mathcal D}}
 \newcommand*{\cL}{{\mathcal L}}
 \newcommand*{\cS}{{\mathcal S}}
 \newcommand*{\cW}{{\mathcal W}}
 \newcommand*{\cX}{{\mathcal X}}
 \newcommand*{\gK}{{\mathfrak K}}
 \newcommand*{\gN}{{\mathfrak N}}
 \newcommand*{\gS}{{\mathfrak S}}
 \newcommand*{\sC}{{\mathsf C}}
 \newcommand*{\sH}{{\mathsf H}}
 \renewcommand*{\ldots}{\mathinner{\ldotp\ldotp\ldotp}}

 \makeatletter
 \newif\ifinany@
 \newcount\column@
 \def\column@plus{%
    \global\advance\column@\@ne
 }
 \newcount\maxfields@
 \def\add@amps#1{%
    \begingroup
        \count@#1
        \DN@{}%
        \loop
            \ifnum\count@>\column@
                \edef\next@{&\next@}%
                \advance\count@\m@ne
        \repeat
    \@xp\endgroup
    \next@
 }
 \def\Let@{\let\\\math@cr}
 \def\restore@math@cr{\def\math@cr@@@{\cr}}
 \restore@math@cr
 \def\default@tag{\let\tag\dft@tag}
 \default@tag

 \newbox\strutbox@
 \def\strut@{\copy\strutbox@}
 \addto@hook\every@math@size{%
  \global\setbox\strutbox@\hbox{\lower.5\normallineskiplimit
         \vbox{\kern-\normallineskiplimit\copy\strutbox}}}

 \renewcommand{\start@aligned}[2]{%
    \RIfM@\else
        \nonmatherr@{\begin{\@currenvir}}%
    \fi
    \null\,%
    \if #1t\vtop \else \if#1b \vbox \else \vcenter \fi \fi \bgroup
        \maxfields@#2\relax
        \ifnum\maxfields@>\m@ne
            \multiply\maxfields@\tw@
            \let\math@cr@@@\math@cr@@@alignedat
        \else
            \restore@math@cr
        \fi
        \Let@
        \default@tag
        \ifinany@\else\openup\jot\fi
        \column@\z@
        \ialign\bgroup
           &\column@plus
            \hfil
            \strut@
            $\m@th\displaystyle{##}$%
           &\column@plus
            $\m@th\displaystyle{{}##}$%
            \hfil
            \crcr
 }
 \renewenvironment{aligned}[1][c]{%
    \start@aligned{#1}\m@ne
 }{%
    \crcr\egroup\egroup
 }
 \makeatother

 \def\cprime{$'$}
 \newif \ifcomments
 \commentsfalse
 \ifcomments \advance\hoffset by-15truemm\fi
 \def\slComm#1#2{\ifcomments%
     {}\vadjust{\kern#1%
      \vtop to 0pt{\vss\hbox to\hsize{\hfill\strut\rlap{ \rm#2}}\null}
      \kern-#1}\fi%
 }
 
 \def\Label#1{\label{#1}\slComm{0mm}{\hbox{\rm #1}}}  
 \def\LabelT#1{\label{#1}\strut\slComm{1mm}{\hbox{\rm #1}}}
 \newcommand*{\Eqref}[1]{\hbox{\rm(\ref{#1})}}

 \newtheorem{thm}{Theorem}[section]
 \newtheorem{cor}[thm]{Corollary}
 \newtheorem{lem}[thm]{Lemma}
 
 \newtheorem*{supp}{Supplement}
 \theoremstyle{definition}
 
 \makeatletter
 \newtheorem*{proofofD.P}{Proof of Theorem~{{\textbf\rm\@newref{thm-D.P}}}} 
 \makeatother 
 \theoremstyle{remark}
 
 \newtheorem{rems}[thm]{Remarks}
 
 \newtheorem{exs}[thm]{Examples}
 \numberwithin{equation}{section}

\begin{document}

%
%
%
%
%
%
%
%
%

\title[Parabolic Equations on Uniformly Regular Riemannian Manifolds]
 {Parabolic Equations on Uniformly Regular Riemannian Manifolds 
 and Degenerate\\ Initial Boundary Value Problems}

\author[Herbert Amann]{Herbert Amann}

\address{%
Math.\ Institut\\ 
Universit\"at Z\"urich\\
Winterthurerstr.~190\\ 
CH 8057 Z\"urich\\ 
Switzerland}

\email{\sf herbert.amann@math.uzh.ch}



\date{January 1, 2004}
\dedicatory{Dedicated to Professor Yoshihiro Shibata on the occasion 
of his sixtieth birthday}

\begin{abstract}
 In this work there is established an optimal existence and regularity 
 theory for second order linear parabolic differential equations on a large 
 class of noncompact Riemannian manifolds. Then it is shown that it provides 
 a general unifying approach to problems with strong degeneracies in the 
 interior or at the boundary. 
\end{abstract}

\maketitle
 \section{Introduction\Label{sec-A}} 
 This paper is devoted to second order initial boundary value problems for 
 linear parabolic equations on a wide class of noncompact Riemannian 
 manifolds, termed `uniformly regular'. Important examples are complete 
 Riemannian manifolds with no boundary and bounded 
 geometry.\footnote{Precise definitions of and notations for all terms used 
 in this introduction without 
 further explanation are found in the following sections and the 
 appendix.} In this setting there is already a rich theory for linear 
 parabolic equations~---~predominantly heat equations~---~based on kernel 
 estimates. Our main interest concerns, however, noncompact Riemannian 
 manifolds with boundary for which very little is known so far (see the 
 following sections for references). Prototypes of such cases are 
 \hbox{$m$-dimensional} Riemannian submanifolds of~$\BR^n$ with compact 
 boundary or funnel-like ends (cf.~Examples~\ref{exs-P.I.ex}). 
 
 \par 
 In order to give the flavor of our main results we consider in this 
 introduction a simplified version of the general problem. Namely, we 
 restrict ourselves to autonomous equations with homogeneous boundary 
 conditions. 
 
 \par 
 Let 
 \hb{M=\Mg} be a Riemannian manifold. We set 
 \beq \Label{A.d} 
 \cA u:=-\tdiv(a\btdot\grad u), 
 \eeq 
 with $a$~being a symmetric positive definite 
 $(1,1)$-tensor field on~$M$ which is bounded and has 
 bounded and continuous first order (covariant) derivatives. This is 
 expressed by saying that $\cA$~is a regular uniformly strongly elliptic 
 differential operator. 
 
 \par 
 We assume that $\pl_0M$~is open and closed in~$\pl M$ and 
 \hb{\pl_1M:=\pl M\ssm\pl_0M}. Then we put 
 $$ 
 \cB_0u:=u\text{ on }\pl_0M 
 \qb \cB_1u:=(\nu\sn a\btdot\grad u)\text{ on }\pl_1M, 
 $$ 
 where these operators are understood in the sense of traces and $\nu$~is 
 the inward pointing unit normal vector field on~$\pl_1M$. Thus 
 \hb{\cB:=(\cB_0,\cB_1)} is the Dirichlet boundary operator on~$\pl_0M$ and 
 the Neumann operator on~$\pl_1M$. 
 
 \par 
 Throughout this paper, 
 \hb{0<T<\iy} and 
 \hb{J:=[0,T]}. We write~$M_T$ for the space time cylinder 
 \hb{M\times J}. Moreover, 
 \hb{\pl=\pl_t} is the `time derivative', 
 \hb{\pl M_T:=\pl M\times J} the lateral boundary, and 
 \hb{M_0=M\times\{0\}} the `initial surface' of~$M_T$. Then we consider the 
 problem  
 \beq \Label{A.P} 
 \pl u+\cA u=f\text{ on }M_T 
 \qb \cB u=0\text{ on }\pl M_T
 \qb u=u_0\text{ on }M_0. 
 \eeq 
 The last equation is to be understood as 
 \hb{\ga_0u=u_0} with the `initial' trace operator~$\ga_0$. 
 
 \par 
 Of course, $\pl_0M$ or $\pl_1M$ or both may be empty. In such a situation 
 obvious interpretations and modifications are to be applied. 
 
 \par 
 We are interested in a strong \hbox{$L_p$-theory} for \Eqref{A.P}. To 
 describe it we have to introduce (fractional order) Sobolev spaces. We 
 always assume that 
 \hb{1<p<\iy}. The Sobolev space~$W_{\coW p}^k(M)$ is defined for 
 \hb{k\in\BN} to be the 
 completion of~$\cD(M)$, the space of smooth functions with compact support, 
 in $L_{1,\loc}(M)$ with respect to the norm 
 \beq \Label{A.N}  
 u\mt 
 \Bigl(\sum_{j=0}^k\big\|\,|\na^ju|_{g_j^0}\,\big\|_{L_p(M)}^p\Bigr)^{1/p}. 
 \eeq 
 Here 
 \hb{\na=\na_{\cona g}} is the Levi-Civita covariant derivative and 
 \hb{\vsdot_{g_j^0}}~the 
 $(0,j)$-tensor norm naturally induced by~$g$. Thus 
 \hb{W_{\coW p}^0(M)=L_p(M)}. Moreover, 
 $$ 
 W_{\coW p,\cB}^2(M):=\bigl\{\,u\in W_{\coW p}^2(M) 
 \ ;\ \cB u=0\,\bigr\}.  
 $$ 
  
 \par 
 We also need the space~$W_{\coW p}^{2-2/p}(M)$ which is defined for 
 \hb{p\neq2} by real interpolation: 
 $$ 
 W_{\coW p}^{2-2/p}(M):=\bigl(L_p(M),W_{\coW p}^2(M)\bigr)_{1-1/p,p}. 
 $$ 
 Then 
 \beq \Label{A.WB} 
  W_{\coW p,\cB}^{2-2/p}(M):= 
 \left\{ 
 \bal 
 &\bigl\{\,u\in W_{\coW p}^{2-2/p}(M)\ ;\ \cB u=0\,\bigr\},
  &\quad 3      &<p<\iy,\\  
 &\bigl\{\,u\in W_{\coW p}^{2-2/p}(M)\ ;\ \cB_0u=0\,\bigr\},
  &\quad 3/2    &<p<3,\\ 
 &W_{\coW p}^{2-2/p}(M),
  &\quad 1      &<p<3/2.\\ 
 \eal 
 \right. 
 \eeq 
 
 \par 
 We set 
 \beq \Label{A.A}  
 A:=\cA\sn W_{\coW p,\cB}^2(M), 
 \eeq 
 considered as an unbounded linear operator in~$L_p(M)$ with 
 domain~$W_{\coW p,\cB}^2(M)$. Then \Eqref{A.P} can be expressed 
 as an initial value problem for the evolution equation 
 \beq \Label{A.J}  
 \thu+Au=f\text{ on }J 
 \qb u(0)=u_0 
 \eeq 
 in~$L_p(M)$. 
 
 \par 
 Now we are ready to formulate our main result in the present model 
 setting. It is a special case of Theorem~\ref{thm-P.S} 
 \begin{thm}\LabelT{thm-A.P} 
 Let $M$ be a uniformly regular Riemannian manifold and let 
 \hb{p\notin\{3/2,3\}}. Suppose that $\cA$~is regular and uniformly strongly 
 elliptic. Then \Eqref{A.P} has for each 
 $$ 
 (f,u_0)\in L_p\bigl(J,L_p(M)\bigr)\times W_{\coW p,\cB}^{2-2/p}(M) 
 $$ 
 a unique solution 
 $$ 
 u\in L_p\bigl(J,W_{\coW p,\cB}^2(M)\bigr) 
 \cap W_{\coW p}^1\bigl(J,L_p(M)\bigr).  
 $$ 
 The map 
 \hb{(f,u_0)\mt u} is linear and continuous.\\ 
 Equivalently: 
 $-A$~generates an analytic semigroup on $L_p(M)$ and 
 has the property of maximal regularity. 
 \end{thm} 
 The finite time interval~$J$ can be replaced by~$\BR^+$, provided we impose 
 the additional assumption that the spectrum of~$A$ is contained in 
 \hb{[\Re z\geq\ga]} for some 
 \hb{\ga>0}. This can always be achieved by replacing~$A$ by 
 \hb{A+\om} for a sufficiently large 
 \hb{\om>0}. 
 
 \par 
 On the surface, this theorem looks exactly the same as the very classical 
 existence and uniqueness theorem for second order parabolic equations on 
 open subsets of~$\BR^m$ with smooth compact boundary 
 (e.g., O.A. Ladyzhenskaya, V.A. Solonnikov, and 
 N.N. Ural'ceva~\cite[Chapter~IV]{LSU68a} and 
 R.~Denk, M.~Hieber, and J.~Pr{\"u}ss~\cite{DHP03a}). However, it is in fact 
 a rather deep-rooted vast generalization thereof since it applies to any 
 uniformly regular Riemannian manifold. 
 
 \par 
 Closely related to uniformly regular Riemannian manifolds are 
 `singular Riemannian manifolds' which are characterized by~a 
 `singularity function' 
 \hb{\rho\in C^\iy\bigl(M,(0,\iy)\bigr)}. More precisely, let 
 \hb{M=\Mg} be a Riemannian manifold and consider the conformal metric 
 \hb{\hat g:=g/\rho^2} on~$M$. Then the basic requirement for~$M$ to be a 
 singular Riemannian manifold is that 
 \hb{\hat M:=\Mhg} be a uniformly regular Riemannian manifold. In 
 Examples~\ref{exs-S.exS} we present some important instances of singular 
 Riemannian manifolds, most notably the class of \hbox{$m$-dimensional} 
 Riemannian submanifolds of~$\BR^n$ with finitely many cuspidal 
 singularities. 
 
 \par 
 By considering parabolic equations on singular Riemannian manifolds we 
 are naturally led to study degenerate parabolic equations in weighted 
 Sobolev spaces. To be more precise, we now assume that 
 \hb{M=\Mg} is a singular Riemannian manifold and 
 \hb{\rho\in C^\iy\bigl(M,(0,\iy)\bigr)} is a singularity function for 
 it. Then $\cA$~is said to be a \hbox{$\rho$-regular} uniformly 
 \hbox{$\rho$-elliptic} differential operator if $\rho^{-2}a$~is 
 symmetric, uniformly positive definite, and $\rho^{-2}a$ 
 and~$\rho^{-1}\na a$ are bounded and continuous. Note that this means 
 that $\cA$~is no longer uniformly strongly elliptic but that the 
 ellipticity condition degenerates if $\rho$~tends to zero (or to infinity). 
 
 \par 
 For 
 \hb{\lda\in\BR} and 
 \hb{k\in\BN} we define the weighted Sobolev 
 space~$W_{\coW p}^{k,\lda}(M;\rho)$ to be the completion of~$\cD(M)$ 
 in~$L_{1,\loc}(M)$ with respect to the norm 
 $$ 
 u\mt 
 \Bigl(\sum_{j=0}^k\bigl\|\rho^{\lda+j}\,|\na^ju|_{g_j^0} 
 \,\bigr\|_{L_p(M)}^p\Bigr)^{1/p}. 
 $$ 
 Then 
 $$ 
 W_{\coW p}^{0,\lda}(M;\rho)=L_p^\lda(M;\rho) 
 :=\bigl\{\,u\in L_{p,\loc}(M)\ ;\ \rho^\lda u\in L_p(M)\,\bigr\}. 
 $$ 
 If 
 \hb{p\neq2}, then 
 $$ 
 W_{\coW p}^{2-2/p,\lda}(M;\rho) 
 :=\bigl(L_p^\lda(M;\rho),W_{\coW p}^{2,\lda}(M;\rho)\bigr)_{1-1/p,p}. 
 $$ 
 Furthermore, $W_{\coW p,\cB}^{2,\lda}(M;\rho)$ 
 and $W_{\coW p,\cB}^{2-2/p,\lda}(M;\rho)$ are defined analogously to 
 $W_{\coW p,\cB}^2(M)$ and $W_{\coW p,\cB}^{2-2/p}(M)$, resp. Lastly, 
 \hb{W_{\coW p}^2(M;\rho):=W_{\coW p}^{2,0}(M;\rho)}, etc. Note that 
 \hb{L_p^0(M;\rho)=L_p(M)}. 
 
 \par 
 Using this we can now formulate our main result for degenerate parabolic 
 equations in the present setting. 
 \begin{thm}\LabelT{thm-A.D} 
 Let $M$ be a singular Riemannian manifold, $\rho$~a singularity function 
 for it, and 
 \hb{p\notin\{3/2,3\}}. Suppose $\cA$~is \hbox{$\rho$-regular} and 
 \hbox{$\rho$-uniformly} strongly elliptic. Then \Eqref{A.P} has for each 
 $$ 
 (f,u_0)\in L_p\bigl(J,L_p(M)\bigr)\times W_{\coW p}^{2-2/p,2/p}(M;\rho) 
 $$ 
 a unique solution 
 $$ 
 u\in L_p\bigl(J,W_{\coW p,\cB}^2(M;\rho)\bigr) 
 \cap W_{\coW p}^1\bigl(J,L_p(M)\bigr).  
 $$ 
 The map 
 \hb{(f,u_0)\mt u} is linear and continuous.\\ 
 Equivalently: Set 
 \hb{A:=\cA\sn W_{\coW p,\cB}^2(M;\rho)}. Then $-A$~generates a strongly 
 continuous analytic semigroup on~$L_p(M)$ and has the property of 
 maximal regularity. 
 \end{thm} 
 This is a particular instance of Theorem~\ref{thm-D.S} and its corollary, 
 both of which apply to general weighted spaces, that is, to 
 \hb{\lda\neq0} as well. 
 
 \par 
 We should like to point out that we impose minimal regularity requirements 
 on~$a$ (within the framework of continuous coefficients). This allows 
 to use Theorems \ref{thm-A.P} and~\ref{thm-A.D} (and the more general 
 results below) as a basis for the study of quasilinear equations along 
 well-established lines (e.g., \cite{Ama95a}, \cite{Ama05a}). For the sake 
 of brevity we do not give details in this paper. 
 
 \par 
 It should also be noted that only the behavior of~$\rho$ near zero and 
 infinity is of importance. In other words, if 
 \hb{\tilde\rho\in C^\iy\bigl(M,(0,\iy)\bigr)} satisfies 
 \hb{\tilde\rho\sim\rho}, that is, 
 \hb{\rho/c\leq\tilde\rho\leq c\rho} for some 
 \hb{c\geq1}, then Theorem~\ref{thm-A.D} remains valid with $\rho$ replaced 
 by~$\tilde\rho$. In particular, $W_{\coW p,\cB}^2(M;\tilde\rho)$ 
 equals $W_{\coW p,\cB}^2(M;\rho)$ except for equivalent norms. 
 
 \par 
 Now we illustrate the strength of our results by means of relatively simple 
 examples. For this we assume that $\Om$~is a smooth open subset 
 of~$\BR^m$ with a compact smooth boundary, that is, $\bar\Om$~is a smooth 
 \hbox{$m$-dimensional} submanifold of~$\BR^m$. We also assume that 
 
 \par  
 \hangindent\parindent 
 $\mf{\Ga}$~is a finite family of compact connected smooth 
 submanifolds~$\Ga$ of~$\BR^m$ without boundary and dimension 
 \hb{\ell_\Ga\leq m-1} such that the following applies: 
 $$ 
  \bal
 \rm{(i)}  \qquad    
    &\text{if $\ell_\Ga=m-1$ and $\Ga\cap\pO\neq\es$, then }\Ga\is\pO.\\
 \rm{(ii)} \qquad    
    &\text{if $1\leq\ell_\Ga\leq m-2$, then }\Ga\is\Om.
 \eal 
 $$
  
 \par  
 \hangindent=0cm 
 \noindent 
 Then 
 \hb{M:=\bar\Om\ssm\bigcup\{\,\Ga\ ;\ \Ga\in\mf{\Ga}\,\}}, endowed with the 
 Euclidean metric, is an \hbox{$m$-dimensional} Riemannian submanifold 
 of~$\BR^m$ whose boundary~$\pl M$ equals 
 \hb{\pO\ssm\bigcup\{\,\Ga\ ;\ \Ga\in\mf{\Ga}\,\}}. 
 
 \par 
 For each 
 \hb{\Ga\in\mf{\Ga}} and 
 \hb{x\in M} we denote by~$\da_\Ga(x)$ the (Euclidean) distance from~$x$ 
 to~$\Ga$. Then $\da_\Ga$~is, sufficiently close to~$\Ga$, a~well-defined 
 strictly positive smooth function. If $M$~contains a neighborhood of 
 infinity in~$\BR^m$, that is, if $\Om$~is an exterior domain, then we put 
 \hb{\da_\iy(x):=|x|} with the Euclidean norm~%
 \hb{\vsdot} in~$\BR^m$. We also fix 
 \hb{\al_\Ga\geq1} and 
 \hb{\al_\iy\in(-\iy,0)}. Then we choose a function 
 \hb{\rho\in C^\iy\bigl(M,(0,\iy)\bigr)} satisfying 
 \hb{\rho\sim\da_\Ga^{\al_\Ga}} near 
 \hb{\Ga\in\mf{\Ga}}, 
 \ \hb{\rho\sim\da_\iy^{\al_\iy}} near infinity if $\Om$~is an exterior 
 domain, and 
 \hb{\rho\sim\mf{1}} away from the `singularity set' 
 \hb{\cS(M):=\bigcup\{\,\Ga\ ;\ \Ga\in\mf{\Ga}\,\}} and infinity. Then 
 $M$~is a singular Riemannian manifold characterized by the singularity 
 function~$\rho$.  Indeed, see Examples~\ref{exs-S.exS}, each 
 \hb{\Ga\in\mf{\Ga}} is an 
 $(\al_\Ga,\ell_\Ga)$-wedge and 
 \hb{\{\,x\in M\ ;\ |x|>R\,\}} is for sufficiently large 
 \hb{R>1} diffeomorphic to an infinite \hbox{$\al_\iy$-cusp} (over~$S^{m-1}$ 
 in~$\BR^{m+1}$) if $\Om$~is an exterior domain. Thus Theorem~\ref{thm-A.D} 
 applies to this situation. 
 
 \par 
 Next we consider some particularly simple subcases which have been treated 
 before in the literature. 
 
 \par 
 (a) Suppose $\Om$~is bounded and 
 \hb{\cS(M)=\pO}. Thus 
 \hb{\rho\sim\da^\al} for some 
 \hb{\al\geq1}, where $\da$~is the distance to~$\pO$. In this situation it is 
 shown by V.~Vespri~\cite{Ves89b} that $A$~generates an analytic semigroup on 
 \hb{L_p(\Om)=L_p(M)}. Recently, S.~For\-naro, G.~Metafune, and 
 D.~Pallara~\cite{FMP11a} have given a new proof for this generation theorem. 
 
 \par 
 (b) Let $\Om$ be bounded and 
 \hb{\ell_\Ga=0} for each 
 \hb{\Ga\in\mf{\Ga}}. Then $\cS(M)$~consists of finitely many one-point sets  
 \hb{\{x_0\},\ldots,\{x_k\}} lying either in~$\Om$ or on~$\pO$. We set 
 \hb{\da_j(x):=|x-x_j|} for 
 \hb{0\leq j\leq k} and 
 \hb{x\in M=\bar\Om\bssm\bigcup_{j=0}^k\{x_j\}}. Assume 
 \hb{\al_j\geq1} for 
 \hb{0\leq j\leq k}. Then Theorem~\ref{thm-A.D} implies that, given any 
 \hb{\rho\in C^\iy\bigl(M,(0,\iy)\bigr)} satisfying 
 \hb{\rho\sim\da_j^{\al_j}} near~$x_j$ and 
 \hb{\rho\sim1} otherwise, 
 \hb{-A=-\cA\sn W_{\coW p,\cB}^2(M;\rho)} generates a strongly continuous 
 analytic semigroup on 
 \hb{L_p(M)=L_p(\Om)} and has the property of maximal regularity. 
 
 \par 
 The only paper known to the author treating the problem of semigroup 
 generation by parabolic equations with strong degeneracies at isolated 
 points is the recent publication of G.~Fragnelli, G.~Ruiz Goldstein, J.A. 
 Goldstein, and S.~Romanelli~\cite{FGGR12a}. These authors consider the 
 case where 
 \hb{\Om=(0,1)} and 
 \hb{\cS(M)=\{x_0\}\is\Om} and show that $-A$~generates an analytic 
 semigroup on~$L_2(\Om)$. 
 
 \par 
 In none of the above papers it is shown that the maximal regularity 
 property prevails. Furthermore, the proofs given there depend 
 significantly on the fact that second order equations are being considered. 
 In contrast, our approach does not depend on the particular structure of 
 the problem but applies equally well to systems and higher order equations 
 (cf.~H.~Amann~\cite{Ama13a}). 
 
 \par 
 Observe that the preceding examples show that a given Riemannian manifold 
 can possess uncountably many non-equivalent singular structures. This is 
 related to and sheds new light on the non-uniqueness results observed by 
 M.A. Pozio, F.~Punzo, and A.~Tesei~\cite{PPT08a}. Thus, besides being 
 rather general and widely applicable, our approach to highly degenerate 
 parabolic problems via Riemannian manifolds leads to a deeper understanding 
 of such problems as well. 
 
 \par 
  In the next section we give the precise definition of a uniformly regular 
 Riemannian manifold. Then we formulate our main result, 
 Theorem~\ref{thm-P.I.P}, in the setting of second order equations 
 and trace it back to the much more general propositions in~\cite{Ama13a}. 
 Note that, besides allowing lower order terms, we prove an optimal 
 regularity theorem in the presence of nonhomogeneous boundary conditions. 
 In addition, we show that we get classical solutions if we impose slightly 
 stronger regularity assumptions on the data. 
 
 \par 
 Singular Riemannian manifolds are precisely defined in Section~\ref{sec-S} 
 and basic examples are presented. Furthermore, weighted function spaces are 
 introduced and their interrelation with non-weighted 
 Sobolev-Slobodeckii spaces on uniformly regular manifolds is established. 
 
 \par 
 Section~\ref{sec-D} contains our main theorem for second order degenerate 
 parabolic problems involving lower order terms and nonhomogeneous 
 boundary conditions. We attract the reader's attention to 
 Theorem~\ref{thm-D.S} where it is shown that problems with homogeneous 
 boundary conditions give rise to generators of analytic semigroups 
 possessing the property of maximal regularity in general weighted 
 spaces~$L_p^\lda(M;\rho)$ for any 
 \hb{\lda\in\BR}. This generalizes results by V.~Barbu, A.~Favini, and 
 S.~Romanelli~\cite{BFR96a}, for example, where the case 
 \hb{M=\Om}, with~$\Om$ a bounded domain on~$\BR^m$, 
 \ \hb{\cS(M)=\pO}, and 
 \hb{\lda=-1} is considered (see also~\cite{FGGR12a}). 
 
 \par 
 For the reader's convenience there is included an appendix in which some 
 basic facts on tensor bundles over Riemannian manifolds are listed. 
 \section{Function Spaces and Uniformly Regular Manifolds\Label{sec-F}}
 By~a \emph{manifold} we always mean a smooth, that is, $C^\iy$~manifold
 with (possibly empty) boundary such that its underlying topological space
 is separable and metrizable. Thus, in the context of manifolds,  we work in
 the smooth category. A~manifold does not need to be connected, but all 
 connected components are of the same dimension. 
 
 \par
 Let 
 \hb{M=\Mg} be a Riemannian manifold with boundary~$\pl M$ and volume 
 measure~$dv$. The 
 metric~$g$ on~$TM$ gives rise to a vector bundle metric on the tensor 
 bundle 
 \hb{V_{\coV\tau}^\sa:=T_\tau^\sa M} for 
 \hb{\sa,\tau\in\BN}, which we denote by~$g_\sa^\tau$ 
 (see the appendix for more details). In particular, 
 \hb{g_1^0=g} and 
 \hb{g_0^1=g^*}, the adjoint (or contravariant) metric on the cotangent 
 bundle~$T^*M$. 
 
 \par 
 For 
 \hb{k\in\BN} the vector space of all 
 \hb{(\sa,\tau)}-tensor fields of class~$C^k$, that is, of all 
 \hbox{$C^k$~sections} of~$V_{\coV\tau}^\sa$, is denoted by 
 $C^k(V_{\coV\tau}^\sa)$. The Levi-Civita covariant  
 derivative,~$\na$, satisfies 
 \hb{\na^ka\in C(V_{\coV\tau+k}^\sa)} for 
 \hb{a\in C^k(V_{\coV\tau}^\sa)}, where 
 \hb{\na^0a:=a}. We write~$\cD(V_{\coV\tau}^\sa)$ for 
 the space of all smooth sections with compact support in~$M$ (which may 
 meet the boundary). As usual, $C^k(M)$ stands for $C^k(V_0^0)$, etc. 
 
 \par 
 We fix $\sa$ and~$\tau$ and set 
 \hb{V:=V_{\coV\tau}^\sa}. Then, given 
 \hb{k\in\BN}, we denote by~$W_{\coW p}^k(V)$ the Sobolev 
 space of order~$k$, defined to be the completion of~$\cD(V)$ in 
 \hb{L_{1,\loc}(V)=L_{1,\loc}(V,dv)} with respect to the norm 
 $$ 
 u\mt 
 \Bigl(\sum_{j=0}^k\bigl\|\,|\na^ju|_{g_\sa^{\tau+j}} 
 \,\bigr\|_{L_p(V)}^p\Bigr)^{1/p}. 
 \npb 
 $$ 
 Thus 
 \hb{W_{\coW p}^0(V)=L_p(V)}. 
 
 \par 
 We also need fractional order Sobolev spaces, namely the Slobodeckii 
 spaces~$W_{\coW p}^s(V)$, for 
 \hb{s\in\BR^+\ssm\BN}. If 
 \hb{k<s<k+1} with 
 \hb{k\in\BN}, then 
 \beq
 \Label{F.I.Slo} 
 W_{\coW p}^s(V) 
 :=\bigl(W_{\coW p}^k(V),W_{\coW p}^{k+1}(V)\bigr)_{s-k,p},
 \eeq 
 which is the interpolation space between $W_{\coW p}^k(V)$ 
 and~$W_{\coW p}^{k+1}(V)$ obtained by means of the real interpolation 
 method with exponent 
 \hb{s-k} and integrability parameter~$p$. 
 
 \par 
 We denote by $B(V)$ the space of all bounded sections of~$V$. It is a Banach 
 space with the norm 
 \hb{u\mt\|u\|_\iy:=\big\|\,|u|\,\big\|_\iy}, where 
 \hb{\Vsdot_\iy} is the maximum norm. Moreover, 
 \hb{BC(V):=B(V)\cap C(V)} is a closed linear subspace thereof. For 
 \hb{k\in\BN} we write $BC^k(V)$ for the linear subspace of~$C^k(V)$ 
 consisting of all~$u$ satisfying 
 \hb{\na^ju\in B(V_{\coV\tau+j}^\sa)} for 
 \hb{0\leq j\leq k}. It is a Banach space with the obvious norm. Moreover, 
 \hb{BC^\iy(V):=\bigcap_kBC^k(V)} and $bc^k(V)$~is the closure 
 of $BC^\iy(V)$ in~$BC^k(V)$. Now we define \emph{Besov-H\"older 
 spaces}~$B_\iy^s(V)$ for 
 \hb{s>0} by 
 \beq\Label{F.B1} 
 B_\iy^s(V):=
 \left\{ 
 \bal 
 &\bigl(bc^k(V),bc^{k+1}(V)\bigr)_{s-k,\iy},
  &\quad k<s    &<k+1,\\ 
 &\bigl(bc^k(V),bc^{k+2}(V)\bigr)_{1/2,\iy},
  &\quad    s   &=k+1, 
 \eal 
 \right. 
 \npb 
 \eeq 
 where 
 \hb{k\in\BN}. 
 
 \par 
 Besides these isotropic spaces we also need anisotropic versions 
 adapted to parabolic problems. Anisotropic Sobolev-Slobodeckii spaces are 
 introduced for 
 \hb{s\in\BR^+=[0,\iy)} by 
 $$ 
 W_{\coW p}^{(s,s/2)}(V\times J) 
 :=L_p\bigl(J,W_{\coW p}^s(V)\bigr) 
 \cap W_{\coW p}^{s/2}\bigl(J,L_p(V)\bigr). 
 $$ 
 The second space on the right is a standard 
 Sobolev-Slobodeckii space of Banach space valued distributions on~$\ci J$. 
 Of course, 
 \hb{W_{\coW p}^{(0,0)}(V\times J)} is naturally identified with 
 \hb{L_p(V\times J)=L_p\bigl(V\times J,dvdt\bigr)}. 
 
 \par 
 Analogously, we define anisotropic Besov-H\"older spaces for 
 \hb{s>0} by  
 $$ 
 B_\iy^{(s,s/2)}(V\times J) 
 :=B\bigl(J,B_\iy^s(V)\bigr)\cap B_\iy^{s/2}\bigl(J,B(V)\bigr). 
 $$ 
 Here the second space on the right is a standard H\"older space 
 $C^{s/2}\bigl(J,B(V)\bigr)$ if 
 \hb{s\notin2\BN}, and a Zygmund space for 
 \hb{s\in2\BN}, of Banach space valued functions on~$J$ 
 (see A.~Lunardi~\cite{Lun95a}, for example). For 
 \hb{k\in\BN^\times:=\BN\ssm\{0\}} we put 
 $$ 
 BC^{(k,k/2)}(V\times J) 
 :=C\bigl(J,BC^k(V)\bigr)\cap C^{k/2}\bigl(J,B(V)\bigr),  
 \npb 
 $$ 
 recalling that $J$~is compact. 
 
 \par 
 Although Sobolev-Slobodeckii spaces, respectively Besov-H\"older spaces, 
 are well-defined for each 
 \hb{s\in\BR^+}, respectively 
 \hb{s>0}, they are not too useful on general Riemannian manifolds 
 since, for example, the fundamental Sobolev type embedding theorems may 
 not hold in general. Even more importantly, there may be no characterization 
 by local coordinates. For this reason we restrict ourselves to the class of 
 \emph{uniformly regular Riemannian manifolds}. Loosely speaking, $M$~is a 
 uniformly regular Riemannian manifold if its differentiable structure is 
 induced by an atlas~$\gK$ of 
 finite multiplicity whose coordinate patches are all of comparable size, 
 such that $\gK$ can be uniformly shrunk to an atlas for~$M$, 
 and the family of all charts in~$\gK$ which intersect~$\pl M$ induces an 
 atlas of the same type for~$\pl M$. In particular, 
 \hb{\pl M=(\pl M,\thg)}, where $\thg$~is the Riemannian metric induced 
 by~$g$ on~$\pl M$, is a uniformly regular 
 \hb{(m-1)}-dimensional Riemannian manifold if $M$~is uniformly regular 
 (see Example~\ref{exs-S.exS}(b)). 
 
 \par 
 For the precise definition of a uniformly regular Riemannian manifold 
 we introduce some notation and conventions. By~$c$ we denote constants
 \hb{\geq1} whose numerical value may vary from occurrence to occurrence;
 but $c$~is always independent of the free variables in a given formula,
 unless a dependence is explicitly indicated. 
 
 \par 
  We denote by~$\BH^m$  the closed right half-space
 \hb{\BR^+\times\BR^{m-1}} in~$\BR^m$, where
 \hb{\BR^0=\{0\}}. The Euclidean metric on~$\BR^m$,
 \hb{(dx^1)^2+\cdots+(dx^m)^2}, is denoted by~$g_m$. The same symbol is used
 for its restriction to an open subset~$U$ of $\BR^m$ or~$\BH^m$, that is,
 for~$\ia^*g_m$, where
 \hb{\ia\sco U\hr\BR^m} is the natural embedding. Here and below, we employ
 standard definitions of pull-back and push-forward operations. 

 \par 
 On the space of all nonnegative functions, defined on some nonempty set 
 whose specific form will be clear in any given situation, 
 we introduce an equivalence relation~%
 \hb{{}\sim{}} by setting
 \hb{f\sim g} iff there exists
 \hb{c\geq1} such that
 \hb{f/c\leq g\leq cf}. Inequalities between vector bundle metrics have to 
 be understood in the sense of quadratic forms. By~$\mf{1}$ we denote the 
 constant function 
 \hb{s\mt1}, whose domain will always be clear from the context. 
 
 \par
 We set
 \hb{Q:=(-1,1)\is\BR}. If $\ka$~is a local chart for an
 \hbox{$m$-dimensional} manifold~$M$, then we write~$U_{\coU\ka}$ for the
 corresponding coordinate patch~$\dom(\ka)$.
  A~local chart~$\ka$ is \emph{normalized} (at~$q$) if
 \hb{\ka(U_{\coU\ka})=Q^m} whenever
 \hb{U_{\coU\ka}\is\ci{M}}, the interior of~$M$, whereas
 \hb{\ka(U_{\coU\ka})=Q^m\cap\BH^m} if
 \hb{U_{\coU\ka}\cap\pl M\neq\es} (and 
 \hb{\ka(q)=0}). We put
 \hb{Q_\ka^m:=\ka(U_{\coU\ka})} if $\ka$~is normalized.
 
 \par 
 An atlas~$\gK$ for~$M$ has \emph{finite multiplicity} if there exists
 \hb{k\in\BN} such that any intersection of more than $k$ coordinate
 patches is empty. In this case 
 $$ 
 \gN(\ka):=\{\,\tk\in\gK 
 \ ;\ U_{\coU\tk}\cap U_{\coU\ka}\neq\es\,\}
 $$
 has cardinality~%
 \hb{\leq k} for each 
 \hb{\ka\in\gK}. An atlas is \emph{uniformly shrinkable} if it consists of
 normalized charts and there exists
 \hb{r\in(0,1)} such that the family 
 \hb{\big\{\,\ka^{-1}(rQ_\ka^m)\ ;\ \ka\in\gK\,\big\}} is a cover of~$M$. 
 We put 
 \hb{\gK_S:=\{\,\ka\in\gK\ ;\ U_{\coU\ka}\cap S\neq\es\,\}} for any 
 nonempty subset~$S$ of~$M$. 

 \par
 Given an open subset~$X$ of $\BR^m$ or~$\BH^m$ and a Banach space~$\cX$, 
 we write
 \hb{\Vsdot_{k,\iy}} for the usual norm of $BC^k\XcX$, the Banach space of
 all
 \hb{u\in C^k\XcX} such that $|\pa u|_\cX$~is uniformly bounded for
 \hb{\al\in\BN^m} of length at most~$k$. 

 \par
 An atlas~$\gK$ for~$M$ is \emph{uniformly regular} if 
 \beq\Label{F.S.K} 
 \bal
 \rm{(i)}  \qquad    &\gK\text{ is uniformly shrinkable and has finite
                      multiplicity.}\\
 \rm{(ii)} \qquad    &\|\tk\circ\ka^{-1}\|_{k,\iy}\leq c(k),
                     \ \ \ka,\tk\in\gK,
                     \ \ k\in\BN.\\
 \eal 
 \eeq 
 In (ii) and in similar situations it is understood that only 
 \hb{\ka,\tk\in\gK} with 
 \hb{U_{\coU\ka}\cap U_{\coU\tk}\neq\es} are being considered. Two 
 uniformly regular atlases $\gK$ and~$\tilde\gK$ are \emph{equivalent}, 
 \hb{\gK\approx\tilde{\gK}}, if 
 \beq\Label{F.S.Keq} 
 \kern-1pt  
 \bal
 \rm{(i)}  \qquad    &\card\{\,\tk\in\tilde{\gK}
                      \ ;\ U_{\coU\tk}\cap U_{\coU\ka}\neq\es\,\}\leq c,
                      \ \ \ka\in\gK;\\
 \rm{(ii)}\qquad    &\|\tk\circ\ka^{-1}\|_{k,\iy}
                     +\|\ka\circ\tk^{-1}\|_{k,\iy}\leq c(k), 
                      \ \ \ka\in\gK, 
                      \ \ \tk\in\tilde{\gK}, 
                      \ \ k\in\BN. 
 \eal
 \kern-1pt  
 \eeq 
 A~uniformly regular structure is a maximal family of equivalent uniformly 
 regular atlases. A~\hh{uniformly regular manifold} is a manifold endowed 
 with a uniformly regular structure. Clearly, on such a manifold all local 
 charts, atlases, etc.\ under consideration belong to its uniformly regular 
 structure. An \hbox{$m$-dimensional} Riemannian manifold $\Mg$ is~a 
 \hh{uniformly regular Riemannian manifold} if 
 \beq\Label{F.S.Rr} 
 \bal
 \rm{(i)}  \qquad    &M\text{ is uniformly regular};\\
 \rm{(ii)} \qquad    &\ka_*g\sim g_m,
                      \ \ \ka\in\gK;\\
 \rm{(iii)}\qquad    &\|\ka_*g\|_{k,\iy}\leq c(k),
                      \ \ \ka\in\gK,
                      \ \ k\in\BN,  
 \eal 
 \npb 
 \eeq 
 for some uniformly regular atlas~$\gK$ for~$M$. 

 \par 
 Let $M$ be a uniformly regular Riemannian manifold. Then 
 the Sobolev-Slobodeckii and Besov-H\"older space scales possess all the 
 properties known to hold in the case of the \hbox{$m$-dimensional} 
 Euclidean space or half-space. In other words, there are embedding, 
 interpolation, and trace theorems for $W_{\coW p}^s(M)$ and 
 $W_{\coW p}^{(s,s/2)}(M\times J)$ which are 
 completely analogous to the corresponding theorems for the classical 
 Sobolev-Slobodeckii spaces. In particular, the anisotropic Sobolev-Morrey 
 type embedding theorem 
 \beq\Label{F.Sob} 
 W_{\coW p}^{(s,s/2)}(V\times J)\hr B_\iy^{(t,t/2)}(V\times J) 
 \qa s+(m+2)/p>t>0, 
 \eeq 
 is valid. In addition, $W_{\coW p}^s(V)$ and~$B_\iy^s(V)$ can be 
 characterized by means of local coordinates, similarly as in the case of 
 compact manifolds. 
 
 \par 
 The spaces $BC^k(V)$ and~$BC^{(k,k/2)}(V)$ do not 
 belong to either one of these scales. However, they can be arbitrarily well 
 approximated by Besov-H\"older spaces. In fact, given 
 \hb{k\in\BN^\times}, 
 \beq\Label{F.B2} 
 B_\iy^{(s_1,s_1/2)}(V\times J)\hr BC^{(k,k/2)}(V\times J) 
 \hr B_\iy^{(s_0,s_0/2)}(V\times J) 
 \eeq 
 for 
 \hb{0<s_0<k<s_1}. Note that this implies a corresponding assertion for the 
 isotropic spaces $BC^k(V)$ and~$B_\iy^s(V)$, since $BC^k(V)$~is naturally 
 identified with the closed linear subspace of $BC^{(k,k/2)}(V)$ of all 
 `time-independent' functions therein, etc. 
 
 \par 
 Proofs, further results, references to related research, and many more 
 details~---~in particular spaces of sections of general uniformly regular 
 vector bundles over~$M$~---~are found in the earlier work 
 \cite{Ama12b}, \cite{Ama12c} of the author (also see \cite{Ama13a}, 
 \cite{Ama13c}, as well as \cite{Ama09a}). 
 \section{Parabolic Problems on Uniformly Regular Riemannian 
 Manifolds\Label{sec-P}} 
 Let 
 \hb{M=\Mg} be a uniformly regular Riemannian manifold. We consider 
 parabolic initial boundary value problems of the form 
 \beq \Label{P.I.P} 
 \pl u+\cA u=f\text{ on }M_T 
 \qb \cB u=h\text{ on }\pl M_T 
 \qb u=u_0\text{ on }M_0. 
 \eeq  
 
 \par 
 In order to reduce the technical apparatus to a minimum we restrict 
 ourselves to the important class of second order divergence form 
 problems. Thus we fix 
 \hb{\da\in C\bigl(\pl M,\{0,1\}\bigr)} and set 
 \hb{\pl_jM:=\da^{-1}(j)} for 
 \hb{j\in\{0,1\}}. Then 
 \hb{\pl M=\pl_0M\cup\pl_1M} and 
 \hb{\pl_0M\cap\pl_1M=\es}. We assume that $\AB$~is of the form 
 \beq \Label{P.I.A} 
 \cA u:=-\tdiv(a\btdot\grad u)+(\vec a\sn\grad u)+a_0u 
 \eeq 
 and 
 $$ 
 \cB u:=
 \left\{ 
 \bal 
 &\cB_0u    &&\quad \text{on }\pl_0M_T,\\ 
 &\cB_1u    &&\quad \text{on }\pl_1M_T,\\ 
 \eal 
 \right. 
 $$ 
 where 
 $$ 
 \cB_0u:=\ga u 
 \qb \cB_1u:=\bigl(\nu\bsn\ga(a\btdot\grad u)\bigr)+b_0\ga u. 
 $$ 
 Here 
 \hb{\prsn=\prsn_g:=g\pr}, \ $\nu$~is the (inward pointing) unit normal 
 on~$\pl M$, \ $\ga$~the trace map for~$\pl M$, and 
 \hb{{}\btdot{}}~denotes complete contraction (see the appendix). More 
 precisely, 
 \hb{\cB_0u=(\ga u)\sn\pl_0M}, etc. We suppose 
 \hb{a\in C^1(T_1^1M\times J)}, \ $\vec a$~is a time-dependent vector field, 
 \ $a_0$ a~function on~$M_T$, and $b_0$ one on ~$\pl_1M_T$. In local 
 coordinates, 
 $$ 
 \cA u=-\frac1{\sqrt{g}}\,\pl_i\bigl(\sqrt{g}\,a_j^ig^{jk}\pl_ku\bigr) 
 +a^i\pl_iu+a_0u. 
 $$ 
 Hence $\cB$~is the Dirichlet boundary operator on~$\pl_0M$ and the 
 Neumann or a Robin boundary operator on~$\pl_1M$. Note that either 
 $\pl_0M$ or~$\pl_1M$ may be empty. We also allow~$M$ to be a manifold 
 without boundary. In this case it is understood throughout the whole 
 paper that all statements, assumptions, and formulas referring explicitly 
 or implicitly to~$\pl M$ are to be unconsidered. For example, 
 problem~\Eqref{P.I.P} reduces to the Cauchy problem 
 $$ 
 \pl u+\cA u=f\text{ on }M_T 
 \qb u=u_0\text{ on }M_0 
 \npb 
 $$ 
 if 
 \hb{\pl M=\es}. 
 
 \par 
 A~function~$u$ satisfying \Eqref{P.I.P} is a \emph{strong} 
 \hbox{$L_p$~\emph{solution}} if it belongs to $W_{\coW p}^{(2,1)}(M_T)$, 
 and~a \emph{classical solution} if it is a member of $BC^{(2,1)}(M_T)$. 
 
 \par 
 The differential operator~$\cA$ is \emph{uniformly strongly elliptic 
 on}~$M_T$ if $a(\cdot,t)$ is symmetric and uniformly positive definite, 
 uniformly with respect to 
 \hb{t\in J}. Clearly, the latter means that there exists a constant 
 \hb{\ve>0} such that 
 $$ 
 \bigl(a(q,t)\btdot X\bsn X\bigr)_{g(q)}\geq\ve\,|X|_{g(q)}^2 
 \qa X\in T_qM 
 \qb q\in M 
 \qb t\in J. 
 $$ 
 
 \par 
 For a concise formulation of the main result we introduce for 
 \hb{s\geq0} the boundary data spaces 
 $$ 
 \bal 
 &W_{\coW p}^{(s+2-\vec\da-1/p)(1,1/2)}(\pl M_T)\\ 
 &\qquad\qquad{} 
  :=W_{\coW p}^{(s+2-1/p)(1,1/2)}(\pl_0M_T) 
  \times W_{\coW p}^{(s+1-1/p)(1,1/2)}(\pl_1M_T), 
 \eal 
 $$ 
 whose general point is written 
 \hb{h=(h_0,h_1)}. Obvious interpretations apply if either $\pl_0M$ 
 or~$\pl_1M$ is empty. The \emph{total data spaces} are then 
 $$ 
 \bal 
 &\cW_{\coW p}^{(s+2,(s+2)/2)}(M_T)\\ 
 &\qquad{} 
  :=W_{\coW p}^{(s,s/2)}(M_T) 
  \times W_{\coW p}^{(s+2-\vec\da-1/p)(1,1/2)}(\pl M_T) 
  \times W_{\coW p}^{s+2-2/p}(M_0) 
 \eal 
 \npb 
 $$ 
 for 
 \hb{s\geq0}. 
 
 \par 
 Given Banach spaces $E$ and~$F$, we denote by $\cL\EF$ the Banach space 
 of bounded linear operators from~$E$ into~$F$. We write $\Lis\EF$ for the 
 subset of all bijections in~$\cL\EF$. Banach's homomorphism theorem 
 guarantees that 
 \hb{A^{-1}\in\cL\FE} if 
 \hb{A\in\Lis\EF}. 
 
 \par 
 Now we can formulate the main existence and uniqueness theorem for 
 problem~\Eqref{P.I.P}. 
 \begin{thm}\LabelT{thm-P.I.P} 
 Let $M$ be a uniformly regular Riemannian manifold and let 
 \hb{p\notin\{3/2,3\}}. Suppose 
 \beq\Label{P.a} 
 \bal 
 \begin{minipage}{300pt}
 $$ 
 a\in BC^{(1,1/2)}(T_1^1M\times J) 
 \qb \vec a\in L_\iy(TM\times J), 
 $$ 
 \end{minipage}\\  
 \begin{minipage}{300pt}
 $$ 
 a_0\in L_\iy(M_T) 
 \qb b_0\in BC^{(1,1/2)}(\pl_1M_T), 
 $$ 
 \end{minipage} 
 \eal 
 \kern-20pt 
 \eeq 
 and $\cA$~is uniformly strongly elliptic. Denote by 
 $\cW_{p,\cc}^{(2,1)}(M_T)$ the vector space of all 
 \hb{(f,h,u_0)\in\cW_p^{(2,1)}(M_T)} satisfying 
 the compatibility conditions of order zero: 
 \beq \Label{P.I.CC} 
 \bal 
 \ga u_0        &=h_0(\cdot,0)  &\text{ on }    &\pl_0M 
 &&\quad\text{if\/ }    &3/2<p  &<3,\\  
 \cB(\cdot,0)u_0&=h(\cdot,0)    &\text{ on }    &\pl M 
 &&\quad\text{if\/ }    &p      &>3. 
 \eal 
 \eeq  
 Then $\cW_{p,\cc}^{(2,1)}(M_T)$ is closed in $\cW_p^{(2,1)}(M_T)$ and 
 \beq \Label{P.WW} 
 (\pl+\cA,\ \cB,\ga_0) 
 \in\Lis\bigl(W_{\coW p}^{(2,1)}(M_T),\cW_{p,\cc}^{(2,1)}(M_T)\bigr). 
 \eeq 
 \end{thm}
 \begin{supp} 
 Suppose 
 \hb{0<s<\ol{s}<1+3/p} with 
 \hb{s\neq3/p} and 
 \beq \Label{P.as} 
 \bal 
 \begin{minipage}{300pt}
 $$ 
 a\in B_\iy^{(1+\ol{s})(1,1/2)}(T_1^1M\times J)  
 \qb \vec a\in B_\iy^{(\ol{s},\ol{s}/2)}(TM\times J), 
 $$
 \end{minipage}\\ 
 \begin{minipage}{300pt} 
 $$ 
 a_0\in B_\iy^{(\ol{s},\ol{s}/2)}(M_T) 
 \qb b_0\in B_\iy^{(1+\ol{s})(1,1/2)}(\pl_1M_T). 
 $$ 
 \end{minipage} 
 \eal 
 \kern-20pt 
 \eeq 
 Let $\cW_{p,\cc}^{(s+2)(1,1/2)}(M_T)$ be the linear subspace of 
 $\cW_p^{(s+2)(1,1/2)}(M_T)$ of all $(f,h,u_0)$ satisfying, in addition 
 to \Eqref{P.I.CC}, the first order compatibility condition 
 \beq \Label{P.CC1} 
 \pl h_0(\cdot,0)+\ga\cA(\cdot,0)u_0=\ga f(\cdot,0)\text{ on }\pl_0M 
 \quad\text{if\/ }s>2/p. 
 \eeq  
 Then $\cW_{p,\cc}^{(s+2)(1,1/2)}(M_T)$ is closed and 
 $$ 
 (\pl+\cA,\ \cB,\ga_0) 
 \in\Lis\bigl(W_{\coW p}^{(s+2)(1,1/2)}(M_T), 
 \cW_{p,\cc}^{(s+2)(1,1/2)}(M_T)\bigr). 
 $$ 
 \end{supp} 
 \begin{proof} 
 We set 
 \hb{a_2:=-a^\sh} and 
 \hb{a_1:=\vec a-\tdiv(a^\sh)}, using the notations of the appendix. Then 
 we get from \Eqref{A.T.dg} 
 \beq \Label{P.A} 
 \cA=a_2\btdot\na^2+a_1\btdot\na+a_0. 
 \eeq  
 We let $\nu_\flat$ be the unit conormal vector field~$g_\flat\nu$ 
 on~$\pl M$ and set 
 \hb{b_1:=\nu_\flat\btdot\ga a^\sh}. Then 
 \beq \Label{P.B} 
 \cB_1=b_1\btdot\ga\na+b_0\ga. 
 \eeq  
 By means of the characterization of $BC^k(T^*M)$ by local coordinates 
 referred to in the preceding section one verifies 
 \beq \Label{P.nu} 
 \nu_\flat\in BC^\iy(T^*M). 
 \eeq  
 
 \par 
 Let \Eqref{P.a} be satisfied. Then it is obvious that 
 $$ 
 a_2\in BC^{(1,1/2)}(T_0^2M\times J) 
 \qb a_1\in L_\iy(TM\times J) 
 \qb a_0\in L_\iy(M_T). 
 $$ 
 Furthermore, \Eqref{P.nu} implies 
 $$ 
 b_1\in BC^{(1,1/2)}\bigl((TM)_{|\pl_1M}\times J\bigr). 
 $$ 
 If \Eqref{P.as} applies, then 
 $$ 
 \bal 
 \begin{minipage}{300pt}
 $$ 
 a_2\in B_\iy^{(1+\ol{s})(1,1/2)}(T_0^2M\times J) 
 \qb a_1\in B_\iy^{(\ol{s},\ol{s}/2)}(TM\times J),
 $$
 \end{minipage}\\ 
 \begin{minipage}{300pt} 
 $$ 
 a_0\in B_\iy^{(\ol{s},\ol{s}/2)}(M_T),  
 $$ 
 \end{minipage} 
 \eal 
 $$ 
 and, once more by \Eqref{P.nu} and the point-wise multiplier result 
 \cite[Theorem~14.3]{Ama12c}, 
 $$ 
 b_1\in B_\iy^{(1+\ol{s})(1,1/2)}\bigl((TM)_{|\pl_1M}\times J\bigr). 
 $$ 
 This shows that $\AB$ satisfies in either case the regularity assumptions 
 of the main theorem of~\cite{Ama13a}. Since the uniform strong 
 ellipticity of~$\cA$ implies that 
 \hb{(\pl+\cA,\ \cB)} is a uniformly strongly parabolic boundary value 
 problem the assertion is a very particular consequence of the latter 
 theorem. 
 \end{proof}
 \begin{cor}\LabelT{cor-P.I.P} 
 Let \Eqref{P.a} be satisfied. Then the initial boundary value problem 
 \Eqref{P.I.P} has for each 
 \hb{(f,h,u_0)\in\cW_{p,\cc}^{(2,1)}(M_T)} a~unique strong 
 $L_p$~solution~$u$ on~$M_T$. Suppose 
 \hb{(m+2)/p<s<1+3/p} with 
 \hb{s\neq3/p}, 
 \,\Eqref{P.as} applies, and 
 \hb{(f,h,u_0)\in\cW_{p,\cc}^{(s+2)(1,1/2)}(M_T)}, then $u$~is a classical 
 solution. 
 \end{cor}
 \begin{proof} 
 The first assertion is clear and the second one follows from 
 \Eqref{F.Sob}. 
 \end{proof} 
 \begin{rems}\LabelT{rems-P.I.P} 
 (a) 
 If 
 \hb{\da=0} (Dirichlet boundary value problem), then 
 \hb{p=3} is admissible as well. If 
 \hb{\da=\mf{1}} (Neumann or Robin boundary conditions), then 
 \hb{p=3/2} can be admitted also. Similarly, \Eqref{P.CC1} is vacuous if 
 \hb{\da=1}. 
 
 \par 
 (b) 
 If all data are smooth and the compatibility conditions of all orders are 
 satisfied, then $u$~is a smooth solution on~$M_T$.
  
 \par 
 (c) 
 We refer to~\cite{Ama13a} for higher order problems and operators acting 
 on sections of general uniformly regular vector bundles over~$M$. 
 
 \par 
 (d) 
 Theorem~\ref{thm-P.I.P} is the basis for establishing results on the 
 existence, uniqueness, and continuous dependence on the data of solutions 
 of quasilinear parabolic problems of the form 
 $$ 
 \pl u+\cA(u)u=F(u)\text{ on }M_T 
 \qb \cB(u)u=H(u)\text{ on }\pl M_T
 \qb u=u_0\text{ on }M_0. 
 $$ 
 Such results are obtained by (more or less obvious) modifications of the 
 proofs in~\cite{Ama05a}. This is to be carried out somewhere 
 else.\hfill$\qed$
 \end{rems} 
 Of course, Theorem~\ref{thm-P.I.P} applies in particular to autonomous 
 problems. To simplify the presentation we restrict ourselves to the setting 
 of strong $L_p$~solutions. Then \Eqref{P.a} reduces to 
 \beq \Label{P.aa} 
 \bal 
 a  &\in BC^1(T_1^1M), 
        &\quad \vec a  &\in L_\iy(TM),\\  
 a_0&\in L_\iy(M), 
        &\quad b_0      &\in BC^1\bigl((TM)_{|\pl_1M}\bigr).  
 \eal 
 \eeq  
 Of particular importance is the case of homogeneous boundary value 
 problems. 
 
 \par  
 Theorem~\ref{thm-P.I.P} guarantees 
 \hb{\cA\in\cL\bigl(W_{\coW p}^2(M),L_p(M)\bigr)}. Hence~$A$, the 
 restriction of~$\cA$ to~$W_{\coW p,\cB}^2(M)$, is a well-defined element 
 of $\cL\bigl(W_{\coW p,\cB}^2(M),L_p(M)\bigr)$. Moreover, 
 $A$~is closed (cf.\ \cite[Lemma~I.1.1.2]{Ama95a}) and densely defined 
 (since $\cD(\ci M)$~is a subset of~$W_{\coW p,\cB}^2(M)$ and 
 $\cD(\ci M)$~is dense in~$L_p(M)$). By means of~$A$ we can reformulate the 
 autonomous homogeneous initial boundary value problem~\Eqref{A.P} as the 
 evolution equation~\Eqref{A.J}. This is made precise by the next theorem 
 for which we rely on semigroup theory and maximal regularity (see H.~Amann 
 \cite[Chapter~III]{Ama95a} 
 and~\cite{Ama04a}, R.~Denk, M.~Hieber, and J.~Pr{\"u}ss~\cite{DHP03a}, or 
 P.Ch. Kunstmann and L.~Weis~\cite{KuW04a}, for example, for 
 information on these concepts). 
 \begin{thm}\LabelT{thm-P.S} 
 Let $M$ be a uniformly regular Riemannian manifold and let 
 \hb{p\notin\{3/2,3\}}. Suppose $\cA$~is autonomous, uniformly strongly 
 elliptic, and conditions~\Eqref{P.aa} are satisfied. Then $-A$~generates 
 a strongly continuous analytic semigroup on~$L_p(M)$ and has the property 
 of maximal regularity, that is to say, 
 \hb{(\pl+A,\ \ga_0)} belongs to  
 \beq \Label{P.Lis} 
 \Lis\bigl( 
 L_p\bigl(J,W_{\coW p,\cB}^2(M)\bigr) 
  \cap W_{\coW p}^1\bigl(J,L_p(M)\bigr),
 L_p(M_T)\times W_{\coW p,\cB}^{2-2/p}(M)\bigr). 
 \eeq 
 \end{thm}
 \begin{proof} 
 In the present setting 
 \hb{\cW_{p,\cc}^{(2,1)}(M_T)=L_p(M_T)\times W_{\coW p,\cB}^{2-2/p}(M)}. 
 Hence \Eqref{P.Lis} is a reformulation of \Eqref{P.WW}. Now the semigroup 
 assertion follows from a result of G.~Dore~\cite{Dor93b}. 
 \end{proof} 
 To indicate the power of these theorems we need to know examples of 
 uniformly regular Riemannian manifolds. This problem is dealt with 
 in~\cite{Ama13d} where proofs for the following claims are found. 
 \begin{exs}\LabelT{exs-P.I.ex} 
 (a) 
 Every compact Riemannian manifold is a uniformly regular Riemannian 
 manifold.

 \par 
 (b) 
 An \hbox{$m$-dimensional} Riemannian submanifold of~$\BR^m$
 possessing a compact boundary is a uniformly regular Riemannian 
 manifold.

 \par
 (c) 
 \hb{\BR^m=(\BR^m,g_m)} and 
 \hb{\BH^m=(\BH^m,g_m)} are uniformly regular Riemannian manifolds. 

 \par 
 (d) 
 Let 
 \hb{\tilde{M}=\tMtg} be a Riemannian manifold and  
 \hb{\vp\sco\Mg\ra\tMtg} an isometry. Then $M$~is a uniformly regular 
 Riemannian manifold iff $\tilde{M}$~is one. 

 \par 
 (e) 
 A~Riemannian manifold has \emph{bounded geometry} if it has no boundary, 
 a~positive injectivity radius, and all covariant derivatives of the 
 curvature tensor are bounded. Every complete Riemannian manifold with 
 bounded geometry is a uniformly regular Riemannian manifold.
 
 \par 
 (f) 
 Suppose 
 \hb{S\is U\is M}, where $S$~is closed and $U$~is open in~$M$. An 
 atlas~$\gK$ for~$U$ is \emph{uniformly regular on}~$S$ if \Eqref{F.S.K} 
 holds with $\gK$ replaced by~$\gK_S$. Two uniformly regular atlases $\gK$ 
 and~$\tilde\gK$ for~$U$ on~$S$ are equivalent if \Eqref{F.S.Keq} applies 
 to $\gK_S$ and~$\tilde\gK_S$. This defines a uniformly regular structure 
 for~$U$ \emph{on}~$S$. Then $U$~is \emph{uniformly regular on}~$S$ if it is 
 endowed with a uniformly regular structure on~$S$. Lastly, $U$~is a 
 \emph{uniformly regular Riemannian manifold on}~$S$ if \Eqref{F.S.Rr} is 
 satisfied for $U$ and~$\gK_S$, where $\gK$~is a uniformly regular atlas 
 for~$U$ on~$S$. 
 
 \par 
 Let 
 \hb{S_j\is U_{\coU j}\is M} and suppose $U_{\coU j}$~is a uniformly regular 
 Riemannian manifold on~$S_j$ for 
 \hb{0\leq j\leq\ell}. Let $\gK_j$ be a uniformly regular atlas 
 for~$U_{\coU j}$ on~$S_j$. Assume  
 
 \par 
 ($\al$) 
 \hb{\|\ka_i\circ\ka_j^{-1}\|_{k,\iy}\leq c(k)}, 
 \ \hb{(\ka_i,\ka_j)\in\gK_{i,S_i}\times\gK_{j,S_j}}, 
 \ \hb{0\leq i,j\leq\ell}, 
 \ \hb{k\in\BN}; 
 
 \par 
 ($\ba$) 
 \hb{M=S_0\cup\cdots\cup S_\ell}.

 \par 
 \noindent 
 Then 
 \hb{\gK:=\gK_0\cup\cdots\cup\gK_\ell} is a uniformly regular atlas for~$M$ 
 and $M$~is a uniformly regular Riemannian manifold. It is obtained by 
 \emph{patching together the uniformly regular pieces~$U_{\coU j}$ on~$S_j$}. 
 
 \par 
 (g) 
 Assume 
 \hb{d\geq m} and $B$~is an 
 \hb{(m-1)}-dimensional compact submanifold of~$\BR^{d-1}$. For a nonempty 
 subinterval~$I$ of~$(1,\iy)$ and 
 \hb{0\leq\al\leq1} we set 
 $$ 
 F_\al\IB:=\bigl\{\,(t,t^\al y)\ ;\ t\in I,\ y\in B\,\bigr\} 
 \is\BR\times\BR^{d-1}=\BR^d, 
 $$ 
 Then 
 \hb{F_\al(B):=F_\al\bigl((1,\iy),B\bigr)}, 
 endowed with the Riemannian metric induced by~$\BR^d$, that is, by~$g_d$, 
 is an \hbox{$m$-dimensional} 
 Riemannian submanifold of~$\BR^d$ with boundary~$F_\al(\pl B)$, where 
 \hb{F_\al(\es):=\es}. It is called \hbox{$\al$-\emph{funnel}} in~$\BR^d$. 
 Note that~a 
 \hbox{$0$-funnel} is a cylinder and~a 
 \hbox{$1$-funnel} a (blunt) cone over (the basis)~$B$. 
 
 \par 
 Let 
 \hb{F=F_\al(B)} be an \hbox{$\al$-funnel} in~$\BR^d$ and set 
 \hb{S:=F_\al\bigl([2,\iy),B\bigr)}. Then $F$~is an \hbox{$m$-dimensional} 
 uniformly regular Riemannian manifold on~$S$. 

 \par 
 (h) 
 Suppose $U$~is open in~$M$ and 
 \hb{F=F_\al(B)} an \hbox{$m$-dimensional} \hbox{$\al$-funnel} in~$\BR^d$. 
 Set 
 \hb{F(I):=F_\al\IB}. Assume 
 \hb{\vp\sco U\ra F} is a diffeomorphism such that 
 \hb{S:=\vp^{-1}\bigl(F[2,\iy)\bigr)} satisfies  
 
 \par 
 ($\al$) 
 \hb{(\,\ol{U\ssm S}\,)\cap S=\vp^{-1}\bigl(F(\{2\})\bigr)}; 
 
 \par 
 ($\ba$) 
 \hb{\vp_*(g\sn S)\sim g_F\sn F[2,\iy)}. 
 
 \par 
 \noindent 
 Then $U$~is a uniformly regular Riemannian manifold on~$S$, and $M$~is 
 said to have an $(\al,B)$-\emph{funnel-like end in~$U$ with 
 representation}~$\vp$.

 \par 
 (i) 
 Let 
 \hb{U_0,\ldots,U_\ell} be open in~$M$. Suppose 
 
 \par 
 ($\al$) 
 \hb{U_i\cap U_{\coU j}=\es,\ \ 1\leq i<j\leq\ell}; 
 
 \par 
 \hangindent\parindent 
 ($\ba$) 
 $M$~has an $(\al_j,B_j)$-funnel-like end in~$U_{\coU j}$ with 
 representation~$\vp_j$ for 
 \hb{j\geq1}; 
 
 \par 
 \hangindent=0cm
 ($\ga$) 
 \hb{\vp_j(U_0\cap U_{\coU j})=F_j(1,4)}, 
 \ \hb{j\geq1}; 
 
 \par 
 ($\da$) 
 \hb{S_0:=U_0\ssm\bigcup_{j=1}^\ell\vp_j^{-1}\bigl(F_j(3,\iy)\bigr)} is 
 compact. 
 
 \par 
 \noindent 
 Then $M$~is a uniformly regular Riemannian manifold, a~\emph{Riemannian 
 manifold with finitely many funnel-like ends}. It is obtained by 
 patching together the uniformly regular pieces~$U_{\coU j}$ on~$S_j$ for 
 \hb{0\leq j\leq\ell}.\hfill$\qed$ 
 \end{exs}
 The most elementary situation in which Theorem~\ref{thm-P.I.P} applies is 
 the case in which $M$~is compact. If, notably, $M$~is the closure of a 
 smooth bounded open subset of~$\BR^m$, then our theorem reduces essentially 
 to a well-known classical result (e.g.,~\cite{LSU68a}). 
 
 \par 
 More recently, G.~Grubb~\cite{Gru95b} has established a general 
 \hbox{$L_p$~theory} for parabolic pseudo-differential boundary value 
 problems acting on sections of vector bundles (also see Section~IV.4.1 
 in~\cite{Gru96a}). It applies to a class of noncompact manifolds, called 
 `admissible' and being introduced in G.~Grubb and 
 N.J.~Kokholm~\cite{GrK93a}. It is a subclass of the above family of 
 manifolds with funnel-like ends, namely a family of manifolds with 
 conical ends. 
 Of course, aside from the requirements on the manifold, differential 
 boundary value problems of the form considered in the present paper 
 constitute a very particular subcase of Grubb's general class. 
 However, in order to apply the results of~\cite{Gru95b} to \Eqref{P.I.P} 
 we have to require that $\AB$ has \hbox{$C^\iy$~coefficients}. In contrast, 
 we impose in essence minimal regularity assumptions on~$\AB$. This is 
 important for the study of quasilinear equations on the basis of the 
 linear theorems proved here. 
 
 \par 
 Now we suppose that $M$~is a noncompact uniformly regular Riemannian 
 manifold not belonging to the class of manifolds with funnel-like ends. 
 This is the case, in particular, if $M$~has no boundary, is complete, 
 and has bounded geometry. There is a tremendous amount of literature 
 on heat equations for such manifolds, most of which is an 
 \hbox{$L_2$~theory} and is concerned with kernel estimates and 
 spectral theory (see, for example, E.B. Davies~\cite{Dav89a} or 
 A.~Grigor'yan~\cite{Gri09a} and the references therein). There are a few 
 papers dealing with (semilinear) parabolic equations on noncompact 
 complete Riemannian manifolds under various curvature assumptions which are 
 based on heat kernel estimates
 (e.g.,~Qi~S. Zhang~\cite{Zha97a},~\cite{Zha00a}, 
 A.L. Mazzucato and V.~Nistor~\cite{MaN06a}, F.~Punzo  
 \cite{Pun11a},~\cite{Pun12a}, C.~Bandle, F.~Punzo, and 
 A.~Tesei~\cite{BPuT12a}).  In all these papers either the top-order part is 
 the Laplace-Beltrami operator or smooth leading order coefficients are 
 required. 
 
 \par 
 Except for a recent paper by Y.~Shao and 
 G.~Simonett~\cite{ShS13a}, the author is not aware of any result on 
 parabolic equations on noncompact manifolds which do not rely on heat 
 kernel techniques, leave alone noncompact manifolds with noncompact 
 boundary. In~\cite{ShS13a} the authors, building on \cite{Ama12c} 
 and~\cite{Ama12b}, establish a 
 H\"older space existence theorem for autonomous nonlinear parabolic 
 equations on uniformly regular manifolds without boundary. As an 
 application they 
 show that the solutions of the Yamabe flow instantaneously regularize and 
 become real analytic in space and time. 
  
 \par 
 A prototypical example to which our results apply is furnished by  an 
 \hbox{$m$-dimensional} Riemannian submanifold 
 \hb{M_\sH=(M_\sH,g_\sH)} of the  hyperbolic 
 space~$\sH^m$ represented by the Poincar\'e model. More specifically, 
 we denote by~$\BB^m$ the open unit ball in~$\BR^m$ with 
 closure~$\bar\BB^m$ and boundary~$S^{m-1}$, the 
 \hb{(m-1)}-sphere. Then 
 \hb{\sH=\sH^m=(\BB^m,g_\sH)}, where 
 \hb{g_\sH=4g_m/(1-|x|^2)^2} for 
 \hb{x\in\BB^m}. If $\pl M_\sH$~is not compact, 
 then we assume that its closure in~$\bar\BB^m$ intersects~$S^{m-1}$ 
 transversally and that this intersection is the boundary of an 
 \hb{(m-1)}-dimensional Riemannian submanifold of~$S^{m-1}$. 
 Informally expressed this means, in particular, that $M_\sH$~`does not 
 collapse at infinity'. 
 
 \par 
 Writing~$\tdiv_\sH$ for~$\tdiv_{g_\sH}$, etc., problem~\Eqref{P.I.P} is on 
 \hb{M_\sH\times J} given by 
 $$ 
 \cA_\sH u:=-\tdiv_\sH(a\btdot\grad_\sH u)+(\vec a\sn\grad_\sH u)_\sH 
 +a_0u 
 $$  
 and 
 $$ 
 \cB_\sH u:=
 \left\{ 
 \bal 
 &\ga u     
  &\quad \text{on } &&\pl_0M_\sH,\\ 
 &\bigl(\nu_{\pl M_\sH}\bsn\ga(a\btdot\grad_\sH u)\bigr)_\sH 
  &\quad \text{on } &&\pl_1M_\sH.\\ 
 \eal 
 \right. 
 $$ 
 Using the fact that $g_\sH$~is conformal to~$g_m$ we can express $\cA_\sH$ 
 and~$\cB_\sH$ in terms of~$g_m$, that is, as differential operators on 
 \hb{M:=(M_\sH,g_m)}. In fact, writing 
 \hb{\tdiv=\tdiv_{g_m}}, etc., we find with 
 \hb{\rho(x):=(1-|x|^2)/2} 
 \beq \Label{P.I.div} 
 \bal 
 \tdiv_\sH(a\btdot\grad_\sH u) 
 &=\rho^m\pl_i(\rho^{2-m}a_j^i\da^{jk}\pl_ku) 
  =\rho^m\tdiv(\rho^{2-m}a\btdot\grad u)\\ 
 &=\tdiv(\rho^2a\btdot\grad u)-m(\rho a\btdot\grad\rho\sn\grad u)  
 \eal 
 \eeq 
 and 
 \hb{(\vec a\sn\grad_\sH u)_\sH=(\vec a\sn\grad u)}. Moreover, 
 \hb{\nu_{\pl M_\sH}=\rho\nu} and, consequently, 
 $$ 
 \bigl(\nu_\sH\bsn\ga(a\btdot\grad_\sH u)\bigr)_\sH 
 =\rho\bigl(\nu\bsn\ga(a\btdot\grad u)\bigr). 
 $$ 
 This shows that the initial boundary value problem 
 $$ 
 \pl_tu+\cA_\sH u=f\text{ on }M_\sH\times J 
 \qb \cB_\sH u=h\text{ on }\pl M_\sH\times J 
 \qb u=u_0\text{ on }M_\sH\times\{0\} 
 $$
 can be seen as a degenerate initial boundary value problem on the 
 `underlying' Euclidean manifold~$M$. Note that $M$~is not a uniformly 
 regular Riemannian manifold, even if 
 \hb{\pl M=\es}, that is, 
 \hb{M=\BB^m}, since it cannot be covered by an atlas~$\gK$ whose 
 coordinate patches are uniformly comparable in size and such that a 
 uniform shrinking of~$\gK$ is still an atlas. 
 \section{Singular Riemannian Manifolds and Weighted Function 
 Spaces\Label{sec-S}} 
 Generalizing the preceding example we are led to the concept of singular 
 Riemannian manifolds. Informally speaking, such a manifold is characterized 
 by~a singularity function 
 \hb{\rho\in C^\iy\bigl(M,(0,\iy)\bigr)} such that the conformal metric 
 \hb{\hat g:=g/\rho^2} gives rise to a uniformly regular Riemannian manifold 
 \hb{\hat M:=\Mhg}. To be precise: 
  
 \par 
 Let $M$ be an \hbox{$m$-dimensional} uniformly regular manifold. A~pair 
 $\rgK$ is~a \emph{singularity datum} for~$M$ if 
 \hb{\rho\in C^\iy\bigl(M,(0,\iy)\bigr)} and $\gK$~is a uniformly regular 
 atlas such that 
 \beq\Label{S.sd} 
 \bal
 \rm{(i)}   \qquad    &\|\ka_*\rho\|_{k,\iy}\leq c(k)\rho_\ka, 
                       \ \ \ka\in\gK,
                       \ \ k\in\BN,\\ 
 \noalign{\vskip-1\jot} 
                      &\text{\ where }
                       \rho_\ka:=\ka_*\rho(0)=\rho\bigl(\ka^{-1}(0)\bigr);\\
 \rm{(ii)}  \qquad    &\rho\sn U_{\coU\ka}\sim\rho_\ka,
                       \ \ \ka\in\gK.
 \eal
 \eeq 
 Two singularity data $\rgK$ and $\trtK$ are \emph{equivalent}, 
 \hb{\rgK\approx\trtK}, if 
 \hb{\rho\sim\tilde{\rho}} and 
 \hb{\gK\approx\tilde{\gK}}. 
 
 \par
 A~\emph{singularity structure},~$\gS(M)$, for~$M$ is a maximal family of
 equivalent singularity data. A~\emph{singularity function} for~$M$ is a
 \hb{\rho\in C^\iy\bigl(M,(0,\iy)\bigr)} such that there exists an
 atlas~$\gK$ with
 \hb{\rgK\in\gS(M)}. The set of all singularity functions is the
 \emph{singularity type} of~$M$. It is convenient to denote it by~%
 \hb{[\![\rho]\!]}, where $\rho$~is one of its representatives. 
 
 \par 
 A~\hh{singular Riemannian manifold of type}~%
 \hb{[\![\rho]\!]} is a Riemannian manifold $\Mg$ such that 
 \beq\Label{S.SM} 
 \bal
 \rm{(i)}   \qquad    &M\text{ is uniformly regular and endowed with}\\  
 \noalign{\vskip-1\jot} 
                     &\text{a singularity structure $\gS(M)$ of singularity 
                      type }[\![\rho]\!];\\
 \rm{(ii)}  \qquad    &(M,g/\rho^2)\text{ is a uniformly regular 
                      Riemannian manifold}.
 \eal
 \eeq 
 This definition is independent of the particular choice of~$\rho$ in the 
 following sense: Let 
 \hb{\trtK\approx\rgK}. Then it follows from \Eqref{F.S.Rr}(ii), (iii) and  
 \Eqref{S.sd} that 
 \hb{(M,g/\tilde{\rho}^2)} is a uniformly regular Riemannian manifold and 
 \hb{g/\tilde{\rho}^2\sim g/\rho^2}. In~\cite{Ama13a} it is shown that 
 \Eqref{S.sd}(i) is equivalent to 
 \beq\Label{S.log} 
 d\log\rho\in BC^\iy(T^*\hat M). 
 \eeq 
 In~\cite{Ama13d} there is carried out a detailed study of singular 
 Riemannian manifolds. We refer the reader to that paper for proofs of the 
 following examples. 
 \begin{exs}\LabelT{exs-S.exS} 
 (a) 
 A~uniformly regular Riemannian manifold is singular of 
 type~%
 \hb{[\![\mf{1}]\!]}, and conversely. 

 \par 
 (b) 
 Let $\Ga$ be a union of connected components of~$\pl M$ and 
 \hb{m\geq2}. We endow~$\Ga$ with the induced Riemannian metric 
 \hb{\thg:=\thia\,^*g}, where  
 \hb{\thia\sco\Ga\hr M} is the natural embedding. Let $\gK$ be a uniformly 
 regular atlas for~$M$. For 
 \hb{\ka\in\gK_\Ga} we set 
 \hb{U_{\coU\ithka}:=\pl U_{\coU\ka}:=U_{\coU\ka}\cap\pl M 
    =U_{\coU\ka}\cap\Ga} and 
 \hb{\thka:=\ia_0\circ(\thia\,^*\ka)\sco U_{\coU\ithka}\ra\BR^{m-1}} with 
 \hb{\ia_0\sco\{0\}\times\BR^{m-1}\ra\BR^{m-1}}, 
 \ \hb{(0,x')\mt x'}. Then 
 \hb{\thgK:=\{\,\thka\ ;\ \ka\in\gK_\Ga\,\}} is a uniformly regular atlas 
 for~$\Ga$, the one \emph{induced by}~$\gK$. 
 
 \par 
 Suppose $\Mg$ is a singular Riemannian manifold of type~%
 \hb{[\![\rho]\!]}. We set 
 \hb{\thrho:=\thia\,^*\rho=\rho\sn\Ga}. Then $(\thrho,\thgK)$ 
 is a singularity datum for~$\Ga$. Thus it defines a singular 
 structure~$\thgS(\Ga)$ for~$\Ga$, the one \emph{induced by}~$\gS(M)$. 
 Furthermore, $(\Ga,\thg/\thrho^2)$ is a singular Riemannian manifold of 
 type~%
 \hb{[\![\thrho]\!]} (and dimension 
 \hb{m-1}). It is always understood that $\Ga$~is endowed with 
 the singular structure induced by the one of~$M$. 

 \par 
 (c)  
 Let $\tMtg$ be a Riemannian manifold and 
 \hb{f\sco M\ra\tilde{M}} an isometric diffeomorphism, that is, 
 \hb{\tilde{g}=f_*g}. Suppose $\Mg$ is singular of type~%
 \hb{[\![\rho]\!]} with singularity datum~$\rgK$. Set 
 \hb{f_*\gK:=\{\,f_*\ka\ ;\ \ka\in\gK\,\}}. Then the pair 
 $(f_*\rho,f_*\gK)$ is a singularity datum for $\tMtg$ and the latter is a 
 singular Riemannian manifold of 
 type~%
 \hb{[\![f_*\rho]\!]}.
 
 \par 
 (d) 
 Suppose 
 \hb{S\is U\is M}, where $S$~is closed and $U$~is open in~$M$. Assume 
 \hb{\rho\in C^\iy\bigl(U,(0,\iy)\bigr)} and $\gK$~is a uniformly regular 
 atlas for~$U$ on~$S$ such that \Eqref{S.sd} holds for~$\gK_S$. Then 
 $\rgK$ is a singularity structure for~$U$ \emph{on}~$S$. Two such 
 singularity structures $\rgK$ and~$\trtK$ are equivalent \emph{on}~$S$ if 
 \hb{\rho\sim\tilde\rho} and $\gK$ and~$\tilde\gK$ are equivalent on~$S$. 
 This defines a singularity structure \emph{for~$U$ on}~$S$ of type~%
 \hb{[\![\rho]\!]}. Then $U$~is a \emph{singular Riemannian manifold on~$S$ 
 of type}~%
 \hb{[\![\rho]\!]} if it is endowed with a singularity structure 
 on~$S$ of type~%
 \hb{[\![\rho]\!]} and 
 \hb{(U,g/\rho^2)} is a uniformly regular Riemannian manifold on~$S$. 
 
 \par 
 Assume 
 \hb{S_j\is U_{\coU j}\is M} and $U_{\coU j}$~is a uniformly regular 
 Riemannian manifold on~$S_j$ of type~%
 \hb{[\![\rho_j]\!]} for 
 \hb{0\leq j\leq\ell}. Let $(\rho_j,\gK_j)$ be a singularity structure 
 for~$U_{\coU j}$ on~$S_j$ and assume that ($\al$) and~($\ba$) of 
 Example~\ref{exs-P.I.ex}(f) apply and 
 $$ 
 \rho_i\sn(S_i\cap S_j)\sim\rho_j\sn(S_i\cap S_j) 
 \qa 0\leq i<j\leq\ell. 
 $$ 
 Then there exists 
 \hb{\rho\in C^\iy\bigl(M,(0,\iy)\bigr)} such that 
 \hb{\rho\sn S_j\sim\rho_j\sn S_j} for 
 \hb{0\leq j\leq\ell} and $\rgK$ is a singularity datum for~$M$, where 
 \hb{\gK=\gK_0\cup\cdots\cup\gK_\ell}. Furthermore, $M$~is a singular 
 Riemannian manifold of type~%
 \hb{[\![\rho]\!]}. It is said to be obtained by \emph{patching together 
 the singular Riemannian manifolds~$U_{\coU j}$ on~$S_j$ of type}~%
 \hb{[\![\rho_j]\!]}. 

 \par 
 (e) 
 Let 
 \hb{d\geq2} and suppose $B$~is a \hbox{$b$-dimensional} compact Riemannian 
 submanifold of~$\BR^{d-1}$. For a nonempty subinterval~$I$ of~$(0,1)$ and 
 \hb{\al\geq1} we set 
 $$ 
 C_\al^d\IB 
 :=\bigl\{\,(t,t^\al y)\ ;\ t\in I,\ y\in B\,\bigr\} 
 \is\BR\times\BR^{d-1}=\BR^d.   
 $$ 
 We endow 
 \hb{C_\al^d(B):=C_\al^d\bigl((0,1),B\bigr)} with the metric 
 induced by~$\BR^d$. It is called 
 \emph{model \hbox{$\al$-cusp} over} (the base)~$B$ in~$\BR^d$. 
 Note that a \hbox{$1$-cusp} is a cone. 
 
 \par 
 For 
 \hb{\ell\in\BN} we set 
 \hb{C_{\al,\ell}^d\IB:=C_\al^d\IB} if 
 \hb{\ell=0}, and 
 $$ 
 C_{\al,\ell}^d\IB:=C_\al^d\IB\times IQ^\ell 
 \qa \ell>0, 
 $$ 
 where 
 \hb{IQ^\ell=\{\,tz\in\BR^\ell\ ;\ t\in I,\ z\in Q^\ell\,\bigr\}}. Then 
 \hb{C_{\al,\ell}^d(B):=C_{\al,\ell}^d\bigl((0,1),B\bigr)} is~a 
 \hb{(1+b+\ell)}-dimensional Riemannian submanifold of 
 \hb{\BR^d\times\BR^\ell=\BR^{d+\ell}}, a~\emph{model 
 $(\al,\ell)$-wedge over}~$B$, also 
 called \emph{model $\ell$-wedge over} $C_\al^d(B)$. Thus every model cusp 
 is a model wedge, a \hbox{$0$-wedge}. Every $(\al,\ell)$-wedge 
 \hb{C=C_{\al,\ell}^d(B)} is a singular Riemannian manifold on 
 \hb{S:=C_{\al,\ell}^d\bigl((0,3/4],B\bigr)} of type~%
 \hb{[\![R_\al]\!]}, where the \emph{cusp characteristic}~$R_\al$ is 
 defined by 
 \hb{R_\al(x):=t^\al} for 
 \hb{x=(t,y,z)\in C} with 
 \hb{y\in Q^\ell}. 

 \par 
 (f) 
 Let $U$ be open in~$M$ and set 
 \hb{C:=C_{\al,\ell}^d(B)} and 
 \hb{S:=C_{\al,\ell}^d\bigl((0,3/4],B\bigr)} with 
 \hb{\ell:=m-1-b\geq0}. Suppose that  
 \hb{\vp\sco U\ra C} is a diffeomorphism such that 
 \hb{(\vp_*g)\sn S\sim g_C\sn S}. Then $U$~is a singular Riemannian manifold 
 on~$\vp^{-1}(S)$ of type~%
 \hb{[\![\vp^*R_\al]\!]}. It is said to be an $(\al,\ell)$-\emph{wedge 
 represented by}~$\vp$. 

 \par 
 (g) 
 Let 
 \hb{d\geq m} and let $B$ be an 
 \hb{(m-1)}-dimensional compact submanifold of~$\BR^{d-1}$. For a nonempty 
 subinterval~$I$ of $(1,\iy)$ and 
 \hb{\al<0} we set 
 $$ 
 K_\al^d\IB 
 :=\bigl\{\,(t,t^\al y)\ ;\ t\in I,\ y\in B\,\bigr\} 
 \is\BR\times\BR^{d-1}=\BR^d. 
 $$ 
 Then 
 \hb{K_\al^d(B):=K_\al^d\bigl((1,\iy),B\bigr)} is 
 considered as an \hbox{$m$-dimensional} Riemannian submanifold of~$\BR^d$, 
 an \emph{infinite \hbox{$\al$-cusp} over}~$B$ in~$\BR^d$. Its cusp 
 characteristic~$R_\al$ is given by 
 \hb{R_\al(x):=t^\al} for 
 \hb{x=(t,t^\al y)\in K_\al^d(B)}. It holds that $K_\al^d(B)$~is a singular 
 Riemannian manifold on 
 \hb{K_\al^d\bigl([2,\iy),B\bigr)} of type~%
 \hb{[\![R_\al]\!]}.

 \par 
 (h)
 Let $U$ be open in~$M$ and let 
 \hb{K:=K_\al^d(B)} be an infinite \hbox{$\al$-cusp} over~$B$ in~$\BR^d$. 
 Set 
 \hb{S:=K_\al^d\bigl([2,\iy),B\bigr)}. Let 
 \hb{\vp\sco U\ra K} be a diffeomorphism satisfying 
 
 \par 
 ($\al$) 
 \hb{\ol{U\ssm\vp^{-1}(S)}\cap S=\vp^{-1}\bigl(K_\al^d(\{2\},B)\bigr)}; 
 
 \par 
 ($\ba$) 
 \hb{(\vp_*g)\sn S\sim g_K\sn S}.

 \par 
 \noindent 
 Then $U$~is a singular Riemannian manifold on~$\vp^{-1}(S)$ of type~%
 \hb{[\![\vp^*R_\al[\!]}. Furthermore, $M$~is said to have in~$U$ an 
 \emph{infinite \hbox{$\al$-cusp}} (more precisely: 
 $(\al,B)$-\emph{cusp}) \emph{represented by}~$\vp$.

 \par 
 (i) 
 Assume that $M$~is an \hbox{$m$-dimensional} Riemannian submanifold 
 for~$\BR^n$ for some 
 \hb{n\geq m}. Then 
 \hb{\cS(M):=\bar M\ssm M}, where $\bar M$~is the closure of~$M$ in~$\BR^n$, 
 is the \emph{singularity set} of~$M$. It is independent of~$n$ since the 
 closure of~$M$ in~$\BR^{\tilde n}$ with 
 \hb{\tilde n>n} and 
 \hb{\BR^n=\BR^n\times\{0\}\is\BR^{\tilde n}} equals $\bar M$ also. 
 
 \par 
 Suppose $\Sa$~is a connected component of~$\cS(M)$ with the following 
 properties: 
 
 \par 
 \hangindent\parindent 
 ($\al$) it is an \hbox{$\ell$-dimensional} compact Riemannian submanifold 
 of~$\BR^n$ without boundary, where 
 \hb{\ell\in\{0,\ldots,m-1\}};  
  
 \par 
 \hangindent\parindent 
 ($\ba$) there exist 
 \hb{\al\geq1}, a~compact 
 \hb{(m-\ell-1)}-dimensional Riemannian submanifold~$B$ of~$\BR^d$ with 
 \hb{d\geq m-\ell}, and for each 
 \hb{p\in\Sa} a~normalized chart~$\Phi_p$ for~$\BR^n$ at~$p$ such that, 
 setting 
 \hb{V_{\coV p}:=\dom(\Phi_h)}, 
 $$ 
 \Phi_p(M\cap V_{\coV p})=C_{(\al,\ell)}^d(B)\times\{0\} 
 \is\BR^{d+\ell}\times\BR^{n-d-\ell}=\BR^n  
 $$ 
 and 
 $$ 
 \Phi_p(\Sa\cap V_{\coV p}) 
 =\bigl(\{0\}\times Q^\ell\bigr)\times\{0\}; 
 $$ 
   
 \par 
 \hangindent\parindent 
 ($\ga$) 
 \hb{U_{\coU p}:=M\cap V_{\coV p}} is an $(\al,\ell)$-wedge represented by 
 \hb{\vp_p:=\Phi_p\sn U_{\coU p}}.
 
 \par 
 \noindent 
 \hangindent=0cm
 Then $M$~is said to possess~a \emph{smooth cuspidal singularity of type} 
 $(\al,\ell)$ (more precisely: $(\al,\ell,B)$) \emph{near}~$\Sa$. 
  
 \par 
 Let 
 \hb{\Sa\is\cS(M)} and assume $M$~has a smooth cuspidal singularity of type 
 $(\al,\ell)$ near~$\Sa$. Also assume that there exist relatively compact 
 open neighborhoods $V$ and~$W$ of~$\Sa$ in~$\BR^n$ with 
 \hb{\bar W\is  V} possessing the following properties: set 
 \hb{U:=V\cap M} and 
 \hb{S:=\bar W\cap M}. Then there is 
 \hb{\rho_\al\in C^\iy\bigl(U,(0,\iy)\bigr)} such that 
 $$ 
 \rho_\al\sn(S\cap U_{\coU p})\sim\vp_p^*R_\al\sn(S\cap U_{\coU p}) 
 \qa p\in\Sa, 
 $$ 
 and $U$~is on~$S$ a singular Riemannian manifold of type~%
 \hb{[\![\rho_\al[\!]}. Loosely speaking: $\Sa$~is then said to be~a 
 \emph{smooth $(\al,\ell)$-wedge}, more precisely, 
 a~\emph{smooth $(\al,\ell,B)$-wedge}. It is a \emph{smooth 
 \hbox{$\al$-cusp}}, respectively $(\al,B)$-\emph{cusp}, if 
 \hb{\ell=0}. Note that near every smooth $(\al,S^{m-1})$-cusp $M$~looks 
 locally like 
 \hb{\BR^m\ssm\{0\}} near~$0$. (In this case we choose 
 \hb{d=m+1}.) 

 \par 
 (j) 
 Let $M$ be an \hbox{$m$-dimensional} Riemannian submanifold of~$\BR^n$ for 
 some 
 \hb{n\geq m}. Suppose: 
 
 \par 
 \hangindent\parindent 
 ($\al$) $\cS(M)$~is compact and for each connected component~$\Sa$ 
 of~$\cS(M)$ there exist 
 \hb{\al_\Sa\geq1}, 
 \ \hb{\ell_\Sa\in\{0,\ldots,m-1\}}, and a compact 
 \hb{(m-\ell-1)}-dimensional submanifold~$B_\Sa$ of~$\BR^d$, where 
 \hb{d\geq m-\ell}, such that $M$~has a smooth 
 cuspidal singularity of type 
 \hb{(\al_\Sa,\ell_\Sa,B_\Sa)} near~$\Sa$;  
  
 \par 
 \hangindent\parindent 
 ($\ba$) there are 
 \hb{k\in\BN} and for each 
 \hb{i\in\{1,\ldots,k\}} an open subset~$U_i$ of~$M$, 
 \ \hb{\al_i\in(-\iy,0)}, an 
 \hb{(m-1)}-dimensional compact Riemannian submanifold~$B_i$ of~$\BR^d$ for 
 some 
 \hb{d\geq m}, and a diffeomorphism 
 \hb{\vp_i\sco U_{\coU j}\ra K_{\al_i}^d(B_j)} such that $M$~has in~$U_i$ 
 an infinite $(\al_i,B_i)$-cusp represented by~$\vp_i$; 
   
 \par 
 \hangindent\parindent 
 ($\ga$) 
 \hb{U_i\cap U_{\coU j}=\es} for 
 \hb{1\leq i<j\leq k}, 
 \ \hb{M\ssm(U_1\cup\cdots\cup U_k)} is relatively compact in~$\BR^n$, and 
 \hb{\cS(M)\ssm(U_1\cup\cdots\cup U_k)=\cS(M)}. 

 \par 
 \noindent 
 \hangindent=0cm
 Then $M$~is said to be a \emph{manifold with cuspidal singularities}. Note 
 that $M$~is relatively compact in~$\BR^n$ if 
 \hb{k=0}. 
 
 \par 
 Every manifold with cuspidal singularities is a singular Riemannian manifold 
 of type~%
 \hb{[\![\rho[\!]}, where 
 \hb{\rho\sim\rho_\al} near a smooth $(\al,\ell)$-wedge 
 \hb{\Sa\is\cS(M)}, 
 \ \hb{\rho\sim\vp_i^*R_{\al_i}} on~$U_i$, 
 \ \hb{1\leq i\leq k}, and 
 \hb{\rho\sim\mf{1}} away from the singularities.\hfill\qed 
 \end{exs} 
 The qualifier `smooth' in the preceding definitions refers to the fact 
 that the bases of the cusps are uniformly regular. If they are singular 
 Riemannian manifolds themselves then we get manifolds with cuspidal 
 corners of various orders. For this we refer to \cite{Ama13d} as well. 
 
 \par 
 B.~Ammann, R.~Lauter, and V.~Nistor~\cite{ALN04a} introduce a class of 
 noncompact Riemannian manifolds, termed Lie manifolds, in order to 
 establish regularity properties of solutions to elliptic boundary value 
 problems on polyhedral domains; also see 
 B.~Ammann, A.D. Ionescu and V.~Nistor~\cite{AIN06a}, 
 B.~Ammann and V.~Nistor~\cite{AmmN07a}, and the references therein, 
 as well as the 
 survey by C.~Bacuta, A.L. Mazzucato, and V.~Nistor~\cite{BMN12a}. These 
 authors use a desingularization technique by which they introduce 
 conformal metrics 
 \hb{g/\rho^2}, where $\rho$~is the distance to the singular set. 
 
 \par 
 Let 
 \hb{M=\Mg} be a singular Riemannian manifold of type~%
 \hb{[\![\rho]\!]}. Then we can apply the results of Sections \ref{sec-F} 
 and~\ref{sec-P} to~$\hat M$, where we have to use 
 \hb{\hat\na:=\na_{\hat g}}, of course. Fortunately, since $\hat g$~is 
 conformal to~$g$ we can express all spaces and operators in terms of~$g$, 
 so that $\hat M$~does not appear in the final results. 
 
 \par 
 Specifically, set 
 \hb{V:=V_{\coV\tau}^\sa} and let 
 \hb{\lda\in\BR}. Then the weighted Sobolev space $W_{\coW p}^{k,\lda}\Vrho$ 
 is for 
 \hb{k\in\BN} the completion of~$\cD(V)$ in~$L_{1,\loc}(V)$ with respect 
 to the norm 
 $$ 
 u\mt\Bigl(\sum_{j=0}^k 
 \big\|\rho^{\lda+j+\tau-\sa}\,|\na^ju|_{g_\sa^{\tau+j}} 
 \,\big\|_{L_p(V)}^p\Bigr)^{1/p}. 
 $$ 
 Weighted Slobodeckii spaces $W_{\coW p}^{s,\lda}\Vrho$ are for 
 \hb{k<s<k+1} again defined by interpolation, that is,  by 
 replacing $W_{\coW p}^\ell(V)$ in \Eqref{F.I.Slo} 
 by $W_{\coW p}^{\ell,\lda}\Vrho$ for 
 \hb{\ell\in\{k,k+1\}}. 
 
 \par 
 Analogously, $B^\lda\Vrho$ is the vector space of all sections~$u$ 
 of~$V$ such that 
 \hb{\rho^{\lda+\tau-\sa}\,|u|_{g_\sa^\tau}\in B(M)}. 
 The norm 
 \hb{u\mt\big\|\rho^{\lda+\tau-\sa}\,|u|_{g_\sa^\tau}\big\|_\iy} makes it 
 a Banach space. If 
 \hb{k\in\BN}, then $B^{k,\lda}\Vrho$ is the Banach space of all 
 \hb{u\in C^k(V)} for which 
 $$ 
 \max_{0\leq j\leq k} 
 \big\|\rho^{\lda+j+\tau-\sa}\,|\na^ju|_{g_\sa^{\tau+j}}\big\|_\iy 
 $$ 
 is finite, endowed with this norm. Furthermore, $bc^{k,\lda}\Vrho$ is 
 the closure of 
 \hb{BC^{\iy,\lda}\Vrho:=\bigcap_kBC^k\Vrho} in 
 $BC^{k,\lda}\Vrho$. Then weighted Besov-H\"older spaces 
 $B_\iy^{s,\lda}\Vrho$ are defined by interpolation in complete analogy to 
 \Eqref{F.B1}. 
 
 \par 
 Weighted \hbox{$L_2$~Sobolev} spaces of this type have been introduced by 
 V.A. Kondrat{\cprime}ev~\cite{Kon67a} in the study of elliptic boundary 
 value problems on domains with singular points. Since then they have 
 been used by numerous authors, predominantly in an Euclidean 
 \hbox{$L_2$~setting}. A~detailed study of the \hbox{$L_p$~case} and 
 references are found in~\cite{Ama12b}. 
 
 \par 
 We set 
 $$ 
 \bal 
 \begin{minipage}{300pt}
 $$ 
 W_{\coW p}^s\Vrho:=W_{\coW p}^{s,0}\Vrho 
 \qb BC^k\Vrho:=BC^{k,0}\Vrho, 
 $$
 \end{minipage}\\ 
 \begin{minipage}{300pt} 
 $$ 
 B_\iy^s\Vrho:=B_\iy^{s,0}\Vrho. 
 $$ 
 \end{minipage} 
 \eal 
 $$ 
 In~\cite{Ama13a} it is proved that 
 \beq\Label{S.WW} 
 \bal 
 \begin{minipage}{300pt}
 $$ 
 W_{\coW p}^s\hV\doteq W_{\coW p}^{s,-m/p}\Vrho 
 \qb BC^k\hV\doteq BC^k\Vrho, 
 $$
 \end{minipage}\\ 
 \begin{minipage}{300pt} 
 $$ 
 B_\iy^s\hV\doteq B_\iy^s\Vrho, 
 $$ 
 \end{minipage} 
 \eal  
 \kern-20pt 
 \npb 
 \eeq 
 where 
 \hb{{}\doteq{}}~means `equal except for equivalent norms' and 
 \hb{\hat V:=T_\tau^\sa\hat M}. 
 In~\cite{Ama13a} it is also shown that 
 \hb{(u\mt\rho^\lda u)} belongs to  
 \beq\Label{S.rho} 
 \Lis\bigl(W_{\coW p}^{s,\lda'+\lda}\Vrho,W_{\coW p}^{s,\lda'}\Vrho\bigr) 
 \cap\Lis\bigl(B_\iy^{s,\lda'+\lda}\Vrho,B_\iy^{s,\lda'}\Vrho\bigr) 
 \eeq 
 for 
 \hb{\lda,\lda'\in\BR}, and 
 \hb{(u\mt\rho^\lda u)^{-1}=(v\mt\rho^{-\lda}v)}. Thus it suffices to study 
 the spaces $W_{\coW p}^s\hV$ and $B_\iy^s\hV$ since by this isomorphism 
 and by \Eqref{S.WW} we can transfer all properties from $W_{\coW p}^s\hV$ 
 onto $W_{\coW p}^{s,\lda}\Vrho$ and from $B_\iy^s\hV$ onto 
 $B_\iy^{s,\lda}\Vrho$. Alternatively, we can refer directly 
 to~\cite{Ama12b}. 
 
 \par 
 Anisotropic weighted Sobolev-Slobodeckii spaces are defined for
 \hb{s\geq0} by 
 $$ 
 W_{\coW p}^{(s,s/2),\lda}(V\times J;\rho) 
 :=L_p\bigl(J,W_{\coW p}^{s,\lda}\Vrho\bigr) 
 \cap W_{\coW p}^{s/2}\bigl(J,L_p^\lda\Vrho\bigr). 
 $$ 
 Analogously, we introduce anisotropic Besov-H\"older spaces for 
 \hb{s>0} by
 $$ 
 B_\iy^{(s,s/2),\lda}\Vrho) 
 :=B\bigl(J,B_\iy^{s,\lda}\Vrho\bigr) 
 \cap B_\iy^{s/2}\bigl(J,B^\lda\Vrho\bigr). 
 $$ 
 Again, we omit the superscript~$\lda$ if it equals zero. It is obvious 
 from the above that all embedding, interpolation, and trace theorems, 
 etc.\ proved in~\cite{Ama12c} carry over to the present setting using 
 natural adaptions. It is also clear that \Eqref{S.WW} implies 
 \beq\Label{S.WWa} 
 W_{\coW p}^{(s,s/2)}\hV\doteq W_{\coW p}^{(s,s/2),-m/p}\Vrho. 
 \eeq 
 Furthermore, 
 \beq\Label{S.rha} 
 (u\mt\rho^\lda u) 
 \in\Lis\bigl(W_{\coW p}^{(s,s/2),\lda'+\lda}\Vrho, 
 W_{\coW p}^{(s,s/2),\lda'}\Vrho\bigr) 
 \npb 
 \eeq 
 for 
 \hb{\lda,\lda'\in\BR} is a consequence of \Eqref{S.rho}. 
 \section{Degenerate Parabolic Problems\Label{sec-D}} 
 In this section we study problem~\Eqref{P.I.P} in the case of singular 
 Riemannian manifolds. It turns out that in this situation 
 Theorem~\ref{thm-P.I.P} leads to an isomorphism theorem for degenerate 
 parabolic initial boundary value problems on weighted Sobolev spaces. 
 
 \par 
 Let 
 \hb{M=\Mg} be a singular Riemannian manifold of type~%
 \hb{[\![\rho]\!]} and 
 \hb{\lda\in\BR}. Similarly as in Section~\ref{sec-P}, we introduce data 
 spaces which are now weighted and \hbox{$\lda$-dependent}. To simplify the 
 presentation we restrict ourselves to the setting of strong 
 $L_p$~solutions. Thus we put 
 $$ 
 \bal 
 &W_{\coW p}^{(2-\vec\da-1/p)(1,1/2),\lda+\vec\da+1/p}(\pl M_T;\thrho)\\ 
 &\qquad{} 
  :=W_{\coW p}^{(2-1/p)(1,1/2),\lda+1/p}(\pl_0M_T;\thrho)  
  \times W_{\coW p}^{(1-1/p)(1,1/2),\lda+1+1/p}(\pl_1M_T;\thrho)
 \eal 
 $$ 
 and 
 $$ 
 \bal 
 &\cW_p^{(2,1),\lda}(M_T;\rho)\\ 
 &{} 
  :=L_p^\lda(M_T;\rho)
  \times W_{\coW p}^{(2-\vec\da-1/p)(1,1/2),\lda+\vec\da+1/p} 
   (\pl M_T;\thrho) 
  \times W_{\coW p}^{2-2/p,\lda}(M;\rho).   
 \eal 
 $$ 
 Similarly as before, $\cW_{p,\cc}^{(2,1),\lda}(M_T;\rho)$ is the 
 linear subspace hereof consisting of all 
 $(f,h,u_0)$ satisfying the compatibility conditions~\Eqref{P.I.CC}. 
 
 \par 
 The differential operator~\Eqref{P.I.A} is \emph{uniformly strongly 
 \hbox{$\rho$-elliptic}} if $a(\cdot,t)$ is symmetric for 
 \hb{t\in J} and there exists a constant 
 \hb{\ve>0} such that 
 \beq \Label{D.ell} 
 \bigl(a(q,t)\btdot X\bsn X\bigr)_{g(q)} 
 \geq\ve\rho^2(q)\,|X|_{g(q)}^2 
 \qa X\in T_qM 
 \qb q\in M 
 \qb t\in J. 
 \eeq  
 
 \par 
 Henceforth, we say that $\AB$ is~a \emph{\hbox{$\rho$-regular} uniformly 
 \hbox{$\rho$-elliptic} boundary value problem on}~$M_T$ if $\cA$~is 
 uniformly strongly \hbox{$\rho$-elliptic}, 
 \beq \Label{D.a} 
 \bal 
 \begin{minipage}{300pt}
 $$ 
 a\in BC^{(1,1/2),-2}(T_1^1M\times J;\rho) 
 \qb \vec a\in L_\iy(TM\times J;\rho), 
 $$ 
 \end{minipage}\\  
 \begin{minipage}{300pt}
 $$ 
 a_0\in L_\iy(M_T), 
 $$ 
 \end{minipage} 
 \eal 
 \kern-20pt 
 \eeq  
 and 
 \beq \Label{D.b} 
 b_0\in BC^{(1,1/2),-1}(\pl_1M_T). 
 \eeq  
 If 
 \hb{\rho=\mf{1}}, then $\AB$ is simply called regularly uniformly elliptic. 
 Note that the first part of \Eqref{D.a} implies 
 \hb{|a|_{g_1^1}\leq c\rho^2}. Using this and the symmetry of $a(\cdot,t)$ 
 we see that \Eqref{D.ell} is equivalent to the existence of 
 \hb{\ve\in(0,1)} with 
 \beq \Label{D.ell1} 
 \ve\rho^2(q)\,|X|_{g(q)}^2 
 \leq\bigl(a(q,t)\btdot X\bsn X\bigr)_{g(q)} 
 \leq\rho^2(q)\,|X|_{g(q)}^2/\ve  
 \npb 
 \eeq 
 for 
 \hb{X\in T_qM}, 
 \ \hb{q\in M}, and 
 \hb{t\in J}.  
 
 \par 
 Now we can formulate the following isomorphism theorem for degenerate 
 parabolic equations. 
 \begin{thm}\LabelT{thm-D.P} 
 Let $M$ be a singular Riemannian manifold of type~%
 \hb{[\![\rho]\!]} and 
 \hb{p\notin\{3/2,3\}}. Suppose that $\AB$ is~a \hbox{$\rho$-regular} 
 uniformly \hbox{$\rho$-elliptic} boundary value problem on~$M_T$ and 
 \hb{\lda\in\BR}. 
 
 \par 
 Then $\cW_{p,\cc}^{(2,1),\lda}(M_T;\rho)$ is closed and 
 $$ 
 (\pl+\cA,\ \cB,\ga_0) 
 \in\Lis\bigl(W_{\coW p}^{(2,1),\lda}(M_T;\rho), 
 \cW_{p,\cc}^{(2,1),\lda}(M_T;\rho)\bigr). 
 $$ 
 \end{thm} 
 The proof of this theorem is given later in this section. First we derive 
 an analogue of Theorem~\ref{thm-P.S}. For this we define 
 $$ 
 W_{\coW p,\cB}^{s,\lda}(M;\rho) 
 \qa s\in[0,2]\ssm\{1/p,1+1/p\}, 
 $$ 
 by replacing $W_{\coW p}^s(M)$ in \Eqref{A.WB} by 
 $W_{\coW p}^{s,\lda}(M;\rho)$. 
 \begin{thm}\LabelT{thm-D.S} 
 Let $M$ be a singular Riemannian manifold of type~%
 \hb{[\![\rho]\!]} and 
 \hb{p\notin\{3/2,3\}}. Suppose $\AB$ is an autonomous 
 \hbox{$\rho$-regular} uniformly \hbox{$\rho$-elliptic} boundary value 
 problem on~$M_T$ and 
 \hb{\lda\in\BR}. Set 
 \hb{A^\lda:=\cA\sn W_{\coW p,\cB}^{2,\lda}(M;\rho)}, considered as an 
 unbounded linear operator in $L_p^\lda(M;\rho)$. Then $-A^\lda$~generates 
 a strongly continuous analytic semigroup on $L_p^\lda(M;\rho)$ and has the 
 property of maximal regularity, that is, 
 \hb{(\pl+\cA,\ \ga)} belongs to 
  $$ 
 \Lis\bigl(L_p(J,W_{\coW p,\cB}^{2,\lda}(M;\rho)) 
 \cap W_{\coW p}^1(J,L_p^\lda(M;\rho)), 
 L_p^\lda(M;\rho)\times W_{\coW p,\cB}^{2-2/p,\lda}(M)\bigr). 
 $$ 
 \end{thm}
 \begin{proof} 
 This follows from Theorem~\ref{thm-D.P} by the arguments which led from 
 Theorem~\ref{thm-P.I.P} to Theorem~\ref{thm-P.S}. 
 \end{proof} 
 \begin{cor}\LabelT{cor-D.S} 
 Set 
 \hb{A:=A^0}. Then $-A$~generates a strongly continuous analytic 
 semigroup on $L_p(M)$ and has the maximal regularity property on $L_p(M)$. 
 \end{cor}
 In order to facilitate the proof of Theorem~\ref{thm-D.P} we precede it 
 with a technical lemma. In this connection we identify~$\rho^\lda$ with the 
 point-wise multiplication operator 
 \hb{u\mt\rho^\lda u}. 
 \begin{lem}\LabelT{lem-D.con} 
 Let $\AB$ be a \hbox{$\rho$-regular} uniformly \hbox{$\rho$-elliptic} 
 boundary value problem on~$M_T$ and 
 \hb{\lda\in\BR}. Then there exists another such pair $\ABs$ such that 
 \beq \Label{D.ABc} 
 \AB\circ\rho^\lda=\rho^\lda\circ\ABs. 
 \eeq  
 \end{lem}
 \begin{proof} 
 (1) 
 Note that 
 \beq \Label{D.con} 
 \AB\circ\rho^\lda=\rho^\lda\circ\AB+\bigl[\AB,\rho^\lda\bigr] 
 \eeq  
 where the commutator 
 $$ 
 \bigl[\AB,\rho^\lda\bigr]u:=\AB(\rho^\lda u)-\rho^\lda\AB u 
 $$ 
 is given by 
 \hb{\bigl([\cA,\rho^\lda],[\cB,\rho^\lda]\bigr)} with 
 $$ 
 [\cA,\rho^\lda]u 
 =-2(a\btdot\grad\rho^\lda\sn\grad u) 
 +\bigl((\vec a\sn\grad\rho^\lda)-\tdiv(a\btdot\grad\rho^\lda)\bigr)u 
 $$ 
 and 
 $$ 
 [\cB,\rho^\lda]u 
 =\bigl(0,(\nu\sn\ga(a\btdot\grad\rho^\lda))u\bigr). 
 $$ 
 For abbreviation, 
 \hb{\gl:=\grad\log}. Then 
 \beq \Label{D.grr} 
 \grad\rho^\lda=\lda\rho^{\lda-1}\grad\rho 
 =\rho^\lda\lda\gl\rho. 
 \eeq  
 We set 
 $$ 
 \bal 
 \begin{minipage}{300pt}
 $$ 
\vec a':=-2\lda a\btdot\gl\rho 
 \qb a_0' 
 :=\lda\bigl((\vec a\sn\gl\rho)-\tdiv(a\btdot\gl\rho)\bigr) 
 -\lda^2(a\btdot\gl\rho\sn\gl\rho),   
 $$ 
 \end{minipage}\\  
 \begin{minipage}{300pt}
 $$ 
 b_0':=\lda\bigl(\nu\bsn\ga(a\btdot\gl\rho)\bigr). 
 $$ 
 \end{minipage} 
 \eal 
 $$ 
 Moreover, 
 $$ 
 \cA'u:=\cA u+(\vec a'\sn\grad u)+a_0'u 
 \qb \cB'u:=\cB u+(0,b_0'\ga u). 
 $$ 
 It follows from \Eqref{D.con} and \Eqref{D.grr} that $\ABs$ satisfies 
 \Eqref{D.ABc}. Hence it remains to show that 
 $(\vec a',a_0',b_0')$ possesses the same regularity properties as 
 $(\vec a,a_0,b_0)$. 
 
 \par 
 (2) 
 We know from \Eqref{S.log} and \Eqref{S.WW} that 
 $$ 
 d\log\rho\in BC^1(T^*\hat M)\doteq BC^1(T^*M;\rho). 
 $$ 
 Using this, 
 \hb{\grad=g^\sh d}, and \Eqref{A.T.nga} it follows 
 \beq \Label{D.gl} 
 \gl\rho=g^\sh\,d\log\rho\in BC^{1,2}(TM;\rho). 
 \eeq  
 It is now an easy consequence of this, the assumptions on $a$ 
 and~$\vec a$, and of \Eqref{A.T.c} that 
 \beq \Label{D.aa} 
 \vec a'\in L_\iy(TM\times J) 
 \qb (\vec a\sn\gl\rho)\in L_\iy(M_T). 
 \eeq  
 From \Eqref{D.gl}, \,\Eqref{A.T.Ca}, and \Eqref{A.T.div1} we infer 
 $$ 
 \bal 
 |\tdiv(a\btdot\gl\rho)| 
 &=|\na(a\btdot\gl\rho)|_{g_1^1}\\ 
 &\leq\rho^{-1}\,|\na a|_{g_1^2}\,\rho\,|\gl\rho|_{g_1^0} 
  +\rho^{-2}\,|a|_{g_1^1}\,\rho^2\,|\na\gl\rho|_{g_1^1}. 
 \eal 
 $$ 
 This guarantees that the second summand  of~$a_0'$ belongs to $L_\iy(M_T)$. 
 Similarly, 
 $$ 
 |(a\btdot\gl\rho\sn\gl\rho)| 
 \leq \rho^{-2}\,|a|_{g_1^1}\,(\rho\,|\gl\rho|_{g_1^0})^2 
 $$ 
 implies that the third summand lies in~$L_\iy(M_T)$ as well. Hence, by the 
 second part of \Eqref{D.aa}, 
 \beq \Label{D.a0} 
 a_0'\in L_\iy(M_T). 
 \eeq  
 
 \par 
 Let $\hat\nu$ be the unit normal vector field of~$\pl\hat M$. In local 
 coordinates, 
 \beq \Label{D.nun} 
 \hat\nu=(\hat g_{11})^{-1/2}\pl/\pl x^1 
 =\rho(g_{11})^{-1/2}\pl/\pl x^1=\rho\nu. 
 \eeq  
 Thus 
 \hb{\nu\in BC^{1,1}(TM_{|\pl M};\rho)}. This implies for the conormal field 
 \beq \Label{D.nu} 
 \nu_\flat=g_\flat\nu\in BC^{1,-1}(T^*M_{|\pl M};\rho). 
 \eeq  
 As above, we derive from \Eqref{D.gl} 
 $$ 
 \ga(a\btdot\gl\rho)\in BC^{(1,1/2)}(TM_{|\pl M};\rho).
 $$ 
 Therefore, by \Eqref{D.nu}, 
 \beq \Label{D.b01} 
 b_0'=\lda\nu_\flat\btdot\ga(a\btdot\gl\rho) 
 \in BC^{(1,1/2),-1}(\pl M_T;\rho). 
 \npb 
 \eeq  
 Now \Eqref{D.aa}, \,\Eqref{D.a0}, and \Eqref{D.b01} imply the assertion. 
\end{proof} 
 \begin{proofofD.P} 
 (1) 
 By the definitions of~$\hat g$ and the gradient we get 
 \beq \Label{D.grad} 
 \rho^{-2}(X\sn\grad_{\hat g}u) 
 =(X\sn\grad_{\hat g}u)_{\hat g} 
 =\dl du,X\dr=(X\sn\grad u) 
 \eeq  
 for any \hbox{$C^1$~function}~$u$ and any vector field~$X$ on~$M$. From 
 this we obtain 
 \beq \Label{D.gr} 
 \grad u=\rho^{-2}\grad_{\hat g}u. 
 \eeq  
 We also note that \hbox{\Eqref{A.T.gst}--\Eqref{A.T.n}} 
 imply 
 \beq \Label{D.st} 
 \vsdot_{\hat g_\sa^\tau}=\rho^{\tau-\sa}\,\vsdot_{g_\sa^\tau} 
 \qa \sa,\tau\in\BN. 
 \eeq  
 
 \par 
 We put 
 \hb{\hat a:=\rho^{-2}a}. Then we infer from \Eqref{D.st} 
 \beq \Label{D.2a} 
 |\hat a|_{\hat g_1^1}=\rho^{-2}\,|a|_{g_1^1}. 
 \eeq  
 Note that 
 \hb{\na\hat a=\rho^{-2}\na a-2(d\log\rho)\btdot\hat a} 
 \,(cf.~\Eqref{D.grr}). Hence, see \Eqref{A.T.c}, 
 $$ 
 \bal 
 \rho\,|\na\hat a|_{g_1^2} 
 &\leq\rho^{-1}\,|\na a|_{g_1^2} 
  +2\rho\,|d\log\rho|_{g_0^1}\,|\hat a|_{g_1^1}\\  
 &=\rho^{-1}\,|\na a|_{g_1^2} 
  +2\,|d\log\rho|_{\hat g_0^1}\,|\hat a|_{g_1^1}, 
 \eal 
 $$ 
 the last equality being a consequence of \Eqref{D.st}. From this, 
 \Eqref{S.log}, \,\Eqref{D.2a}, and the assumption on~$a$ we deduce 
 \beq \Label{D.ah} 
 \hat a\in BC^{(1,1/2)}(T_1^1M\times J;\rho) 
 \doteq BC^{(1,1/2)}(T_1^1\hat M\times J). 
 \eeq  
 By replacing the index~$\sH$ in \Eqref{P.I.div} by~$\hat g$ and using 
 \Eqref{D.grad} and \Eqref{D.gr} we find 
 \beq \Label{D.div} 
 \bal 
 \tdiv(a\btdot\grad u) 
 &=\tdiv(\rho^2\hat a\btdot\grad u)\\ 
 &=\tdiv_{\hat g}(\hat a\btdot\grad_{\hat g}u) 
  +m(\rho\hat a\btdot\grad\rho\sn\grad u)\\ 
 &=\tdiv_{\hat g}(\hat a\btdot\grad_{\hat g}u) 
  +(m\hat a\btdot\rho^{-1}\grad_{\hat g}\rho\sn\grad_{\hat g} u)_{\hat g}. 
 \eal 
 \eeq  
 Observe that 
 $$ 
 \rho^{-1}\grad_{\hat g}\rho 
 =\rho^{-1}\hat g^\sh\,d\rho 
 =\hat g^\sh\,d\log\rho. 
 $$ 
 Hence, by \Eqref{A.T.c} and \Eqref{A.T.nga}, 
 $$ 
 |\hat a\btdot\rho^{-1}\grad_{\hat g}\rho|_{\hat g} 
 \leq|\hat a|_{\hat g_1^1}\,|d\log\rho|_{\hat g_0^1}. 
 $$ 
 This, \Eqref{D.ah}, and \Eqref{S.log} imply 
 \beq \Label{D.ar} 
 \hat a\btdot\rho^{-1}\grad_{\hat g}\rho\in L_\iy(T\hat M\times J). 
 \eeq  
 Furthermore, by \Eqref{D.gr} and 
 \hb{\hat g=\rho^{-2}g}, 
 $$ 
 (\vec a\sn\grad u) 
 =(\vec a\sn\rho^{-2}\grad_{\hat g}u) 
 =(\vec a\sn\grad_{\hat g}u)_{\hat g}.  
 $$ 
 From \Eqref{D.st} we derive 
 \beq \Label{D.al} 
 \big\|\,|\vec a|_{\hat g}\,\big\|_\iy 
 =\big\|\rho^{-1}\,|\vec a|_g\,\big\|_\iy 
 =\|\vec a\|_{L_\iy(TM\times J;\rho)}.  
 \eeq  
 Thus it follows from \Eqref{D.ar} that 
 \beq \Label{D.d} 
 \hat d:=m\hat a\btdot\rho^{-1}\grad_{\hat g}\rho+\vec a 
 \in L_\iy(T\hat M\times J). 
 \eeq  
 Now we put 
 $$ 
 \hat\cA u:=-\tdiv_{\hat g}(\hat a\btdot\grad_{\hat g}u) 
 +(\hat d\sn\grad_{\hat g}u)_{\hat g}+a_0u. 
 $$ 
 Then \Eqref{D.div} shows 
 \beq \Label{D.AA} 
 \cA u=\hat\cA u 
 \qa u\in W_{\coW p}^{(2,1),-m/p}(M_T;\rho) 
 \doteq W_{\coW p}^{(2,1)}(\hat M_T), 
 \npb 
 \eeq  
 the last equivalence being a consequence of \Eqref{S.WW}. 
 
 \par 
 (2) 
 Recalling \Eqref{D.nun} and 
 \hb{\thrho=\rho\sn\pl M}, we find 
 \beq \Label{D.nga} 
 \bal 
 \bigl(\nu\bsn\ga(a\btdot\grad u)\bigr) 
 &=\bigl(\nu\bsn\ga(\hat a\btdot\grad_{\hat g}u)\bigr)\\ 
 &=\thrho^2\bigl(\nu\bsn\ga(\hat a\btdot\grad_{\hat g}u)\bigr)_{\hat g} 
  =\thrho\bigl(\hat\nu\bsn\ga(\hat a\btdot\grad_{\hat g}u)\bigr)_{\hat g}. 
 \eal 
 \eeq  
 We denote by~%
 \hb{{}\thna{}} the Levi-Civita connection of~$\pl M$ and set 
 \hb{\hat b_0:=\thrho^{-1}b_0}. Then 
 \beq \Label{D.b0N} 
 \thna\hat b_0=\thrho^{-1}\thna b_0-(d\log\thrho)\hat b_0. 
 \eeq  
 Relation~\Eqref{D.st} implies 
 $$ 
 \thrho\,|d\log\thrho|_{\ithg_0^1} 
 =|d\log\thrho|_{\hat\ithg_0^1}.  
 $$ 
 Since 
 \hb{b_0\in BC^{(1,1/2),-1}(\pl_1M_T;\rho)} it holds 
 \beq \Label{D.bb} 
 \bal 
 \begin{minipage}{300pt}
 $$ 
 \|\hat b_0\|_{L_\iy(\pl_1\hat M_T)} 
 =\|\thrho^{-1}b_0\|_{L_\iy(\pl_1M_T)}<\iy, 
 $$ 
 \end{minipage}\\  
 \begin{minipage}{300pt}
 $$ 
 \big\|\,|\thna b_0|_{\ithg_0^1}\,\big\|_{L_\iy(\pl_1M_T)}<\iy. 
 $$ 
 \end{minipage} 
 \eal 
 \kern-20pt 
 \eeq  
 Thus we get from \Eqref{D.b0N}, \,\Eqref{D.bb}, Example~\ref{exs-S.exS}(b), 
 and \Eqref{S.log} \,(applied to 
 \hb{\pl_1M=\pl_1\hat M}) that 
 \hb{\big\|\thrho\,|\thna\hat b_0|_{\ithg_0^1}\,\big\|_{L_\iy(\pl_1M_T)}} 
 is finite. This implies 
 \beq \Label{D.b02}
 \hat b_0\in BC^{(1,1/2)}(\pl_1M_T;\rho) 
 \doteq BC^{(1,1/2)}(\pl_1\hat M_T). 
 \eeq  
 Now we set 
 $$ 
 \hat\cB_1u 
 :=\bigl(\hat\nu\bsn\ga(\hat a\btdot\grad_{\hat g}u)\bigr)_{\hat g} 
 +\hat b_0\ga u 
 $$ 
 and 
 \hb{\hat\cB:=(\cB_0,\hat\cB_1)}. Then we see from \Eqref{D.ah}, 
 \,\Eqref{D.d}, \,\Eqref{D.b02}, and \Eqref{D.ell} that $\hAB$ is a 
 regular uniformly elliptic boundary value problem on~$\hat M_T$. 
 Furthermore, 
 \beq \Label{D.BBr} 
 \cB u=(\hat\cB_0u,\thrho\hat\cB_1u). 
 \eeq  
 
 \par 
 (3) 
 Suppose 
 \hb{(f,h,u_0)\in\cW_{p,\cc}^{(2,1),-m/p}(M_T;\rho)}. Then, by \Eqref{S.rha}, 
 $$ 
 \thrho^{-1}h_1\in W_{\coW p}^{(1-1/p)(1,1/2),-(m-1)/p}(\pl_1M_T;\thrho). 
 $$ 
 More precisely, set 
 \hb{\hat h:=(h_0,\thrho^{-1}h_1)}. Then \Eqref{S.rha},\ \,\Eqref{S.WW}, 
 and \Eqref{S.WWa} imply 
 $$ 
 \bigl((f,h,u_0)\mt(f,\hat h,u_0)\bigr) 
 \in\Lis\bigl(\cW_{p,\cc}^{(2,1),-m/p}(M_T;\rho), 
 \cW_{p,\cc}^{(2,1)}(\hat M_T)\bigr). 
 $$ 
 In addition, we deduce from \Eqref{D.AA} and \Eqref{D.BBr} that 
 \hb{u\in W_{\coW p}^{(2,1)}(M_T;\rho)} is a solution of \Eqref{P.I.P} iff 
 \hb{u\in W_{\coW p}^{(2,1)}(\hat M_T)} and 
 \hb{(\pl+\hat\cA,\ \hat\cB,\ga_0)u=(f,\hat h,u_0)}. Now 
 Theorem~\ref{thm-P.I.P} implies the validity of the assertion if 
 \hb{\lda=-m/p}. 
 
 \par 
 (4) 
 Let 
 \hb{\lda\neq-m/p}. Lemma~\ref{lem-D.con} guarantees the existence of~a 
 \hbox{$\rho$-regular} uniformly \hbox{$\rho$-elliptic} boundary value 
 problem $\ABs$ on~$M_T$ such that
 \beq \Label{D.AB} 
 \rho^{\lda}\circ\AB=\ABs\circ\rho^\lda. 
 \eeq  
 By \Eqref{S.rho} and \Eqref{S.rha} it follows that 
 \hb{(f,h,u_0)\in\cW_{p,\cc}^{(2,1),\lda}(M_T;\rho)} iff 
 $$ 
 (f',h',u_0'):=\rho^{\lda+m/p}(f,h,u_0) 
 \in\cW_{p,\cc}^{(2,1),-m/p}(M_T;\rho) 
 $$   
 and 
 \hb{u\in W_{\coW p}^{(2,1),\lda}(M_T;\rho)} iff 
 $$ 
 u':=\rho^{\lda+m/p}u\in W_{\coW p}^{(2,1),-m/p}(M_T;\rho). 
 $$ 
 From \Eqref{D.AB} we get 
 $$ 
 (\cA,\cB,\ga_0)u=(f,h,u_0) 
 \ \Llr\  
 (\cA',\cB',\ga_0)u'=(f',h',u_0). 
 \npb 
 $$ 
 As the claim holds for 
 \hb{\lda=-m/p}, the assertion follows.\hfill$\qed$
 \end{proofofD.P}
 \begin{rems}\LabelT{rems-D.P} 
 (a) 
 It is obvious from this proof that there is a straightforward 
 parameter-dependent analogue of the supplement to Theorem~\ref{thm-P.I.P} 
 for degenerate parabolic problems. 
 
 \par 
 (b) 
 Remarks~\ref{rems-P.I.P} apply in the present setting also.\hfill$\qed$
 \end{rems} 
 \section*{Appendix: Tensor Bundles\Label{sec-A.T}} 
 \renewcommand{\theequation}{\rm A.\arabic{equation}} 
 \setcounter{equation}{0} 
 Let $M$ be a manifold and 
 \hb{V=(V,\pi,M)} a vector bundle of rank~$n$ over it. For 
 a nonempty subset~$S$ of~$M$ we denote by~$V_{|S}$ the 
 restriction~$\pi^{-1}(S)$ of~$V$ to~$S$. If $S$~is a submanifold or a 
 union of connected components of~$\pl M$, then $V_{|S}$~is a vector bundle 
 of rank~$n$ over~$S$. As usual, 
 \hb{V_p:=V_{\{p\}}} is the fibre~$\pi^{-1}(p)$ of~$V$ over~$p$. 
 By $\Ga\SV$ we mean the \hbox{$\BR^S$~module} of all sections of~$V$ 
 (\hh{no} smoothness). 
  
 \par 
 As usual, $TM$ and~$T^*M$ are the tangent and cotangent bundles of~$M$. 
 Then 
 \hb{T_\tau^\sa M:=TM^{\otimes\sa}\otimes T^*M^{\otimes\tau}} is for 
 \hb{\sa,\tau\in\BN} the $(\sa,\tau)$-tensor bundle of~$M$, that is, the 
 vector bundle of all tensors on~$M$ being contravariant of order~$\sa$ 
 and covariant of order~$\tau$. In particular,
 \hb{T_0^1M=TM} and 
 \hb{T_1^0M=T^*M}, as well as 
 \hb{T_0^0M=M\times\BR}.

 \par
 For
 \hb{\nu\in\BN^\times} we put
 \hb{\BJ_\nu:=\{1,\ldots,m\}^\nu}. Then, given local coordinates
 \hb{\ka=(x^1,\ldots,x^m)} and setting
 $$
 \frac\pl{\pl x^{(i)}}
 :=\frac\pl{\pl x^{i_1}}\otimes\cdots\otimes\frac\pl{\pl x^{i_\sa}}
 \qb dx^{(j)}:=dx^{j_1}\otimes\cdots\otimes dx^{j_\tau}
 $$
 for
 \hb{(i)=(i_1,\ldots,i_\sa)\in\BJ_\sa},
 \ \hb{(j)\in\BJ_\tau}, the local representation of~a $(\sa,\tau)$-tensor 
 field 
 \hb{a\in\Ga(T_\tau^\sa M)} with respect to these coordinates is given by
 $$ 
 a=a_{(j)}^{(i)}\frac\pl{\pl x^{(i)}}\otimes dx^{(j)}
 $$ 
 with
 \hb{a_{(j)}^{(i)}\in\BR^{U_{\coU\ka}}}. We use the summation convention 
 for \hbox{(multi-\nobreak)}\linebreak[0]indices
 labeling coordinates or 
 bases. Thus such a repeated index, which appears once as a superscript and 
 once as a subscript, implies summation over its whole range. 

 \par 
 Suppose 
 \hb{\sa_1,\sa_2,\tau_1,\tau_2\in\BN}. Then the 
 \emph{complete contraction}  
 $$ 
 \Ga(T_{\tau_2+\sa_1}^{\sa_2+\tau_1}M) 
 \times\Ga(T_{\tau_1}^{\sa_1}M) 
 \ra\Ga(T_{\tau_2}^{\sa_2}M) 
 \qb (a,b)\mt a\btdot b  
 $$ 
 is defined as follows: Given 
 \hb{(i_k)\in\BJ_{\sa_k}} and  
 \hb{(j_k)\in\BJ_{\tau_k}} for 
 \hb{k=1,2}, we set 
 $$ 
 (i_2;j_1):=(i_{2,1},\ldots,i_{2,\sa_2},j_{1,1},\ldots,j_{1,\tau_1}) 
 \in\BJ_{\sa_2+\tau_1} 
 $$ 
 etc., using obvious interpretations if 
 \hb{\min\{\sa,\tau\}=0}. Suppose 
 \hb{a\in\Ga(T_{\tau_2+\sa_1}^{\sa_2+\tau_1}M)} and 
 \hb{b\in\Ga(T_{\tau_1}^{\sa_1}M)} are locally represented 
 on~$U_{\coU\ka}$ by 
 $$ 
 a=a_{(j_2;i_1)}^{(i_2;j_1)}\,\frac\pl{\pl x^{(i_2)}} 
 \otimes\frac\pl{\pl x^{(j_1)}} 
 \otimes dx^{(j_2)}\otimes dx^{(i_1)} 
 \qb  b=b_{(j_1)}^{(i_1)}\,\frac\pl{\pl x^{(i_1)}} 
 \otimes dx^{(j_1)}. 
 $$ 
 Then the local representation of 
 \hb{a\btdot b} on~$U_{\coU\ka}$ is given by 
 $$ 
 a_{(j_2;i_1)}^{(i_2;j_1)}\,b_{(j_1)}^{(i_1)}\,\frac\pl{\pl x^{(i_2)}} 
 \otimes dx^{(j_2)}. 
 $$ 

 \par 
 Let $g$ be a Riemannian metric on~$TM$. We write
 \hb{g_\flat\sco TM\ra T^*M} for the (fiber-wise defined) Riesz 
 isomorphism. Thus
 \hb{\dl g_\flat X,Y\dr=g\XY} for 
 \hb{X,Y\in\Ga(TM)}, where
 \hb{\pw\sco\Ga(T^*M)\times\Ga(TM)\ra\BR^M} 
 is the natural (fiber-wise defined) duality pairing. The inverse 
 of~$g_\flat$ is 
 denoted by~$g^\sh$. Then~$g^*$, the adjoint Riemannian metric on~$T^*M$, 
 is defined by 
 \hb{g^*(\al,\ba):=g(g^\sh\al,g^\sh\ba)} for 
 \hb{\al,\ba\in\Ga(T^*M)}. In local coordinates 
 \beq \Label{A.T.gg1} 
 g=g_{ij}\,dx^i\otimes dx^j 
 \qb g^*=g^{ij}\,\frac\pl{\pl x^i}\otimes\frac\pl{\pl x^j}, 
 \npb 
 \eeq 
 $[g^{ij}]$~being the inverse of the 
 \hb{(m\times m)}-matrix~$[g_{ij}]$. 
 
 \par 
 The metric~$g$ induces a vector bundle metric on~$T_\tau^\sa M$ which 
 we denote by~$g_\sa^\tau$. In local coordinates 
 \beq \Label{A.T.gst} 
 g_\sa^\tau(a,b)=g_{(i)(j)}g^{(k)(\ell)}a_{(k)}^{(i)}b_{(\ell)}^{(j)}
 \qa a,b\in\Ga(T_\tau^\sa M), 
 \eeq 
 where 
 \beq \Label{A.T.gg2} 
 g_{(i)(j)}:=g_{i_1j_1}\cdots g_{i_\sa j_\sa}
 \qb g^{(k)(\ell)}:=g^{k_1\ell_1}\cdots g^{k_\tau\ell_\tau} 
 \eeq 
 for
 \hb{(i),(j)\in\BJ_\sa} and
 \hb{(k),(\ell)\in\BJ_\tau}. Note 
 \hb{g_1^0=g} and 
 \hb{g_0^1=g^*} and 
 \hb{g_0^0(a,b)=ab} for 
 \hb{a,b\in\Ga(M\times\BR)=\BR^M}. Moreover, 
 \beq \Label{A.T.n} 
 \vsdot_{g_\sa^\tau}\sco\Ga(T_\tau^\sa M)\ra(\BR^+)^M 
 \qb a\mt\sqrt{g_\sa^\tau(a,a)} 
 \eeq
 is the \emph{vector bundle norm} on~$T_\tau^\sa M$ induced by~$g$. 
 It follows that the complete contraction satisfies 
 \beq \Label{A.T.c} 
 |a\btdot b|_{g_{\sa_2}^{\tau_2}} 
 \leq|a|_{g_{\sa_2+\tau_1}^{\tau_2+\sa_1}}\,|b|_{g_{\sa_1}^{\tau_1}}  
 \qa a\in\Ga(T_{\tau_2+\sa_1}^{\sa_2+\tau_1}M) 
 \qb b\in\Ga(T_{\tau_1}^{\sa_1}M). 
 \eeq  

 \par 
 We define a vector bundle isomorphism 
 \hb{T_{\tau+1}^\sa M\ra T_\tau^{\sa+1}M}, 
 \ \hb{ a\mt a^\sh} by 
 \beq \Label{A.T.aa} 
 a^\sh(\al_1,\ldots,\al_\sa,\al,X_1,\ldots,X_\tau) 
 :=a(\al_1,\ldots,\al_\sa,X_1,\ldots,X_\tau,g^\sh\al) 
 \eeq 
 for 
 \hb{X_1,\ldots,X_\tau\in\Ga(TM)} and 
 \hb{\al,\al_1,\ldots,\al_\sa\in\Ga(T^*M)}. If $a_{(j;k)}^{(i)}$ with 
 \hb{(i)\in\BJ_\sa}, 
 \ \hb{(j)\in\BJ_\tau}, and 
 \hb{k\in\BJ_1} is the coefficient of~$a$ in a local coordinate 
 representation, then 
 \beq \Label{A.T.ga} 
 (a^\sh)_{(j)}^{(i;k)}=g^{k\ell}a_{(j;\ell)}^{(i)}. 
 \eeq  
 This implies 
 \beq \Label{A.T.nga} 
 |a^\sh|_{g_{\sa+1}^\tau}=|a|_{g_\sa^{\tau+1}}. 
 \eeq  
 
 \par 
 The Levi-Civita connection on~$TM$ is denoted by 
 \hb{\na=\na_{\cona g}}. We use the same symbol for its natural extension 
 to a metric connection on~$T_\tau^\sa M$. Then the corresponding covariant 
 derivative is the linear map 
 $$
 \na\sco C^\iy(T_\tau^\sa M)\ra C^\iy(T_{\tau+1}^\sa M) 
 \qb a\mt\na a,
 $$
 defined by
 \hb{\dl\na a,b\otimes X\dr:=\dl\na_{\cona X}a,b\dr} for 
 \hb{b\in C^\iy(T_\sa^\tau M)} and  
 \hb{X\in C^\iy(TM)}. It is a well-defined continuous linear 
 map from~$C^1(T_\tau^\sa M)$ into~$C(T_{\tau+1}^\sa M)$, as follows from 
 its local representation. For 
 \hb{k\in\BN} we define 
 $$ 
 \na^k\sco C^k(T_\tau^\sa M)\ra C(T_{\tau+k}^\sa M) 
 \qb a\mt\na^ka 
 \npb 
 $$ 
 by
 \hb{\na^0a:=a} and
 \hb{\na^{k+1}:=\na\circ\na^k}. 
 
 \par 
 In local coordinates 
 \hb{\ka=(x^1,\ldots,x^m)} the volume measure 
 \hb{dv=dv_g} of $\Mg$ is represented by 
 \hb{\ka_*\,dv=\ka_*\sqrt{g}\,dx}, where 
 \hb{\sqrt{g}:=\bigl(\det[g_{ij}]\bigr)^{1/2}} and $dx$~is the Lebesgue 
 measure on~$\BR^m$. 
 
 \par 
 The \hb{contraction} 
 \hb{\sC\sco T_{\tau+1}^{\sa+1}M\ra T_\tau^\sa M}, 
 \ \hb{a\mt\sC a} is given in local coordinates by 
 \hb{(\sC a)_{(j)}^{(i)}:=a_{(j;k)}^{(i;k)}}. It follows 
 \beq \Label{A.T.Ca} 
 |\sC a|_{g_\sa^\tau}=|a|_{g_{\sa+1}^{\tau+1}}. 
 \eeq  
 Recall that the divergence of tensor fields is the map  
 \beq \Label{A.T.div1} 
 \tdiv=\tdiv_{\cotdiv g}\sco C^1(T_\tau^{\sa+1}M) 
 \ra C(T_\tau^\sa M) 
 \qb a\mt\tdiv a:=\sC(\na a). 
 \eeq 
 If $X$~is a $C^1$~vector field on~$M$, then 
 $\tdiv X$ has the well-known local representation 
 \beq \Label{A.T.div2} 
 \frac1{\sqrt{g}}\,\frac\pl{\pl x^i}\bigl(\sqrt{g}\,X^i\bigr) 
 \qa X=X^i\,\frac\pl{\pl x^i}. 
 \eeq 
 The gradient, 
 \hb{\grad u=\grad_gu}, of a \hbox{$C^1$~function}~$u$ is the continuous 
 vector field $g^\sh du$. 
 
 \par 
 Suppose 
 \hb{a\in C^1(T_1^1M)}. Then, in terms of covariant derivatives, 
 \beq \Label{A.T.dg} 
 \tdiv(a\grad u) 
 =a^\sh\btdot\na^2u+\tdiv(a^\sh)\btdot\na u. 
 \eeq  
 
\def\cprime{$'$} \def\polhk#1{\setbox0=\hbox{#1}{\ooalign{\hidewidth
  \lower1.5ex\hbox{`}\hidewidth\crcr\unhbox0}}}

\end{document}